\documentclass[12pt]{article}
\usepackage{graphicx} % Required for inserting images
\usepackage[top=.9in, bottom=.9in, left=.5in , right=.5in]{geometry}
\usepackage{amsmath,amsthm,amssymb, mathrsfs}
\usepackage[mathlines]{lineno}
\usepackage{comment}
\newtheorem{thm}{Theorem}[section]
\newtheorem*{thm*}{Theorem}
\newtheorem{lem}[thm]{Lemma}
\newtheorem{clm}{Claim}[thm]
\newtheorem{prop}[thm]{Proposition}
\newtheorem{cor}[thm]{Corollary}

\usepackage{xcolor}

\theoremstyle{definition}
\newtheorem{rmq}[thm]{Remark}

\newcommand{\E}[2][]{\ensuremath{\mathbb{E}_{#1}\left[#2 \right]}}
\newcommand{\Prob}[2][]{\ensuremath{\mathbb{P}_{#1} \left(#2 \right)}}
\newcommand{\Var}[2][]{\operatorname{Var}_{#1}\left(#2\right)}
\newcommand{\eps}{\varepsilon}
\newcommand{\dd}{\mathrm{d}}
\newcommand{\Probtwo}[2][]{\ensuremath{\mathbb{Q}_{#1} \left(#2 \right)}}
\DeclareMathOperator{\spine}{spine}

\newcommand{\ensymboldremark }{\hfill$\blacktriangleleft$}
\title{Growth scales and uniform integrability of branching processes in varying environments}
\author{Tejas Iyer\footnote{Weierstrass Institute for Applied Analysis and Stochastics, Anton-Willem-Amo Str. 39, 10117 Berlin, Germany.}}
\date{\today}

\begin{document}
\maketitle
\abstract{We derive necessary and sufficient conditions for a sequence of constants $(C_{n})_{n \in \mathbb{N}_0}$ to be a growth scale of a branching process in (possibly defective) varying environments $(Z_{n})_{n \in \mathbb{N}_0}$ in the sense that $(Z_{n}/C_{n})$ is bounded from above and below with positive probability. Along the way, we derive a Kesten-Stigum type result - necessary and sufficient conditions for branching processes with varying environments to converge in $L^1$ when normalised by their mean. The proofs exploit truncation and change of measure arguments, convergence of random series and martingale techniques.}


\section{Introduction}

Branching processes in varying environments (BPVEs) are among the simplest extensions of the classical Bienaymé–Galton–Watson (BGW) process. As in the classical setting, offspring numbers are independent across individuals, but the reproduction law is allowed to vary from one generation to the next. Unlike finite-mean BGW processes, where the population grows on an essentially exponential scale on survival, BPVEs can exhibit a wide spectrum of growth regimes. This makes BPVEs a more realistic toy model for populations evolving under changing conditions, where exponential growth is typically only a transient phenomenon before external constraints become significant. A natural question is therefore to identify deterministic normalisations that capture the eventual size of the population of a BPVE.

Much of the existing theory (e.g.~\cite{schuh-barbour-infinite-mean, dsouza-rates-of-growth}) addresses this problem through \emph{rates of growth}: deterministic sequences $(C_n)_{n\in\mathbb{N}_0}$ for which $Z_n/C_n$ converges to a non-zero finite limit with positive probability. In a varying environment, however, it is not clear that the correct order of magnitude should be accompanied by convergence of the normalised process, since the reproduction law itself varies with $n$. The environment may itself introduce deterministic oscillations or changes of scale, while rare exceptionally large offspring values may disrupt a proposed normalisation altogether. This motivates the weaker, but more robust, notion of a \emph{growth scale}. Instead of requiring convergence of $Z_n/C_n$, we ask only that the normalised process remain bounded away from both $0$ and $\infty$ along the full sequence, with positive probability. 

In this article we provide necessary and sufficient conditions for a deterministic sequence $(C_n)_{n\in\mathbb{N}_0}$ to be a growth scale of a (possibly defective) branching process in varying environment. The criterion separates the two issues suggested above. Offspring values much larger than the proposed scale must occur only finitely often when the population is of order $C_n$; after truncating these large reproduction events, the mean of the resulting finite-mean process must be comparable to $C_n$; and the associated martingale must converge in $L^1$. We characterise this final condition by proving a Kesten--Stigum type theorem for BPVEs. In particular, the result shows that $(C_n)_{n\in\mathbb{N}_0}$ is a growth scale for the original process if and only if the corresponding sequence of truncated means is a rate of growth for the truncated finite-mean process. 

The proofs are based in part, on a spine argument for the Kesten--Stigum type theorem, and a diagonalisation approach to definng a Heyde martingale. This gives a more transparent route to the criterion and avoids much of the inverse-generating-functions and analytical machinery used in earlier treatments~\cite{schuh-barbour-infinite-mean, dsouza-rates-of-growth} of general growth rates. In this sense, the paper not only extends the classical rate-of-growth viewpoint to the more flexible notion of growth scales, but also provides an insight into probabilistic mechanism underlying the corresponding conditions.

\subsection{Model and background} 
Formally, suppose we are given a collection of laws $\boldsymbol{\mu} = (\mu_{n})_{n \in \mathbb{N}_0}$ on the space $\overline{\mathbb{N}}_{0} : = \mathbb{N} \cup \{0\} \cup \{\infty\}$, and a collection of random variables $(X_{n, j})_{n \in \mathbb{N}_0, j \in \mathbb{N}}$ that are independent, such that each $X_{n, j}$ has law $\mu_{n}$. A branching process in varying environments (BPVE) $(Z_{n})_{n \in \mathbb{N}_0}$, with offspring laws $(\mu_{n})_{n \in \mathbb{N}_0}$ is the stochastic process such that $Z_{0} = 1$ and, 
\begin{equation} \label{eq:definition-of-process}
    Z_{n+1} = \sum_{j=1}^{Z_{n}} X_{n, j}. 
\end{equation}    
We use $X_n$ to denote a generic random variable with law $\mu_n$. The value $\infty$ is absorbing, i.e., once $Z_n=\infty$, we set $Z_m=\infty$ for every $m\geq n$. In the terminology of the branching-process literature, such processes are often called \emph{defective}; see, for example,~\cite{defective-bgw, defective-bpves}.

\begin{rmq}
    The classical Bienaym\'e-Galton-Watson (BGW) is the special case in which each $X_n$ takes values in $\mathbb{N}_0$ and all offspring laws are equal. 
\end{rmq}

A sequence $(C_{n})_{n \in \mathbb{N}_0} \in [1, \infty)^{\mathbb{N}_0}$ is said to be a \emph{rate of growth} of a BPVE $(Z_{n})_{n \in \mathbb{N}_0}$ if, assuming the limit exists with positive probability, 
\[
\Prob{0 < \lim_{n \to \infty} \frac{Z_{n}}{C_{n}} < \infty} > 0. 
\]
For the constant sequence $C_{n} \equiv 1$, Church~\cite{church} provided necessary and sufficient conditions for a non-defective BPVE to converge in distribution to a non-zero, finite random variable. Lindvall~\cite{lindval} showed that, in the non-defective case, every BPVE converges almost surely. 

Another natural candidate rate of growth is the mean. Suppose that 
\begin{equation} \label{eq:possibly-new-equation}
0 < m_{n} := \E{X_{n}} < \infty 
\end{equation}
for every $n \in \mathbb{N}_0$, and define 
\begin{equation} \label{eq:mean-bpve}
M_{n} := \begin{cases}
    \prod_{j=0}^{n-1} \E{X_{j}} & \text{if $n \geq 1$} \\
    1 & \text{otherwise}. 
\end{cases}
\end{equation}
Then, $(Z_{n}/M_{n})$ is a non-negative martingale, hence converges almost surely to a finite limit, and the question of whether $(M_{n})_{n \in \mathbb{N}_0}$ is a rate of growth is a question about the non-degeneracy of this limit. In the case of BGW processes, this is answered by the Kesten-Stigum theorem~\cite{kesten-stigum, conceptual-xlogx}: the mean is a rate of growth if and only if $\E{X \log{X}} < \infty$. For BPVEs sufficient criteria for uniform integrability of the martingale were proved by Fearn~\cite{fearn} and Jagers~\cite{Jagers} using $L^2$ boundedness. These were improved by~\cite{goettge-suff}, and later obtained as a special case of the change-of-measure framework of Biggins and Kyprianou~\cite{measure-change}. 

Seneta-Heyde norming~\cite{seneta, Heyde}, and analysis of Heyde's martingale in the infinite mean case~\cite{schuh-barbour-infinite-mean} provide necessary and sufficient conditions for a general sequence $(C_{n})_{n \in \mathbb{N}_0}$ to be a rate of growth of a classical BGW process. In the varying environment setting, MacPhee--Schuh~\cite{macphee-schuh} and D'Souza~\cite{dsouza-rates-of-growth}, gave criteria for a general BPVE to have a rate of growth $(C_{n})_{n \in \mathbb{N}_0}$. These works also show that BPVEs can have several, even infinitely many, non-trivial (i.e. non-asymptotically equivalent) rates of growth. 

The approach in~\cite{dsouza-rates-of-growth, macphee-schuh} is based on Heyde-type martingales. While this is natural in the homogeneous BGW setting where one can iterate inverse generating functions~\cite{schuh-barbour-infinite-mean}, the corresponding machinery is technically delicate in varying environments: the relevant generating functions, their ranges, and the domains on which inverses are defined all vary with the generation. In~\cite{macphee-schuh, dsouza-rates-of-growth}, these varying environment adaptations are invoked rather than proved in detail. We therefore provide a self-contained proof. Our proof reduces the growth-scale problem to a finite-mean truncated process and then applies a spine criterion for $L^1$ convergence. Generating functions still enter in one step, but the argument avoids inverses by constructing a Heyde martingale via a diagonalisation argument - a construction which, to our knowledge, is novel. 

Other works on BPVEs include criteria for the existence of a single rate of growth~\cite{supercritical-xlogx-biggins, varying-environments-seneta-heyde, kirpicheva2026branchingprocessvaryingenvironment}, critical behaviour~\cite{yaglom-limit-law}, nearly degenerate behaviour~\cite{nearly-degenerate}, genealogical and coalescent structure~\cite{geneological-structure, coalescent-structure}, and a general setting where classifications into subcritical, critical and supercritical regimes are possible~\cite{classification}.

\subsection{Main results}
Our first main result is a Kesten-Stigum type theorem for finite-mean BPVEs. For a non-negative, integer valued random variable $X$ with $0 < \E{X} < \infty$, let $\widehat{X}$ denote the associated size-biased version of $X$, defined by
\begin{equation}
\Prob{\widehat{X} = k} = \frac{k \Prob{X = k}}{\E{X}}. 
\end{equation} 
\begin{thm} \label{thm:unif-integrable}
    Let $(Z_{n})_{n\in \mathbb{N}_0}$ be a BPVE with offspring sequence $\boldsymbol{\mu}$ satisfying $\E{X_{n}} < \infty$ for all $n \in \mathbb{N}$. Suppose $(M_{n})_{n \in \mathbb{N}_0}$ is as defined in~\eqref{eq:mean-bpve} and let $K \in (0,\infty)$. Then,  
     \begin{enumerate}
            \item \label{item:non-deg-limit} The martingale $(Z_{n}/M_{n})_{n \in \mathbb{N}_{0}}$ converges in $L^1$ if and only if 
            \begin{equation} \label{eq:limsup-mean}
            \limsup_{n \to \infty} 1/M_{n} := M^{*} < \infty,
            \end{equation}
            and
            \begin{equation} \label{eq:l1-requirement}
                \sum_{n=0}^{\infty} \Prob{\widehat{X}_{n} - 1\geq K M_{n+1}} + \frac{\E{\left(\widehat{X}_{n} - 1 \right)\mathbf{1}_{\widehat{X}_n - 1 \leq K M_{n+1}}}}{M_{n+1}} < \infty.  
            \end{equation}
            Moreover,  
            \begin{equation}
            \label{eq:prob-pos}
                \Prob{\lim_{n \to \infty} \frac{Z_{n}}{M_{n}} > 0} \geq \E{\frac{1}{M^{*} + \sum_{n=0}^{\infty} \frac{\widehat{X}_{n} - 1}{M_{n+1}}}} > 0.
            \end{equation}
            \item \label{item:deg-limit} If~\eqref{eq:l1-requirement} fails then
             \[\lim_{n \to \infty} \frac{Z_{n}}{M_n} = 0 \quad \text{almost surely,}\] whilst if~\eqref{eq:limsup-mean} fails we have $\lim_{n \to \infty} Z_n =0$ almost surely. 
        \end{enumerate}
\end{thm}

\begin{rmq}
    By the Kolmogorov three-series theorem, condition~\eqref{eq:l1-requirement} is equivalent to  
    \begin{equation} \label{eq:expl-1} \sum_{m=0}^{\infty} \frac{\widehat{X}_{m} - 1}{M_{m+1}} < \infty \quad \text{almost surely.}
    \end{equation}
    For the classical BGW process, Equation~\eqref{eq:expl-1} is equivalent to $\E{X \log{X}} < \infty$; see also~\cite{lyons-pemantle-peres}.   {\small\ensymboldremark }
\end{rmq}

\begin{rmq}
    The sufficient direction of Theorem~\ref{thm:unif-integrable} first appeared in~\cite[Theorem 5]{goettge-suff}, and, via a spine argument in~\cite[Corollary~3.1]{measure-change}. {\small\ensymboldremark }
\end{rmq}

\begin{rmq} \label{rmq:novelty-unif}
The main point in the proof of Theorem~\ref{thm:unif-integrable} is the treatment of the degenerate regime in Item~\ref{item:deg-limit}. The classical spine argument identifies the summability of
\[
\sum_{n=0}^{\infty}\frac{\widehat X_n-1}{M_{n+1}}
\]
as the condition for a non-degenerate martingale limit. The additional difficulty is to prove the converse in the regime where, for some $K>0$,
\[
\sum_{n=0}^{\infty}
\Prob{\widehat X_n-1\geq K M_{n+1}}<\infty
\]
but the corresponding truncated expectation sum diverges. {\small\ensymboldremark }
\end{rmq}

For our next result, note that a deterministic sequence may describe the correct order of magnitude of a BPVE even when the normalised process does not converge. This motivates the following definition. We say that $(C_n)_{n\in\mathbb{N}_0}\in [1,\infty)^{\mathbb{N}_0}$ is a \emph{growth scale} of $(Z_n)_{n\in\mathbb{N}_0}$ if
\begin{equation}\label{eq:growth-scale-defn}
\Prob{
0<\liminf_{n\to\infty}\frac{Z_n}{C_n}
\leq
\limsup_{n\to\infty}\frac{Z_n}{C_n}
<\infty}>0.
\end{equation}
Thus every rate of growth is a growth scale, but a growth scale allows persistent oscillations in the normalised population size. 

\begin{thm} \label{thm:rate-of-growth}
    Suppose $(Z_{n})_{n\in \mathbb{N}_0}$ is a BPVE with offspring sequence $\boldsymbol{\mu}$ and $(C_{n})_{n \in \mathbb{N}_0} \in \mathbb{N}^{\mathbb{N}_0}$. Then, the sequence $(C_{n})_{n \in \mathbb{N}_0}$ is a growth scale if and only if there exists $\gamma \in (0, \infty)$ such that
        \begin{equation} \label{eq:truncation-determines-original}
            \sum_{n=0}^{\infty} C_{n} \Prob{X_{n} > \gamma C_{n+1}} < \infty,
        \end{equation}
        and, with $Y^{(\gamma)}_{n}$ distributed like $X_{n} \mathbf{1}_{X_{n} \leq \gamma C_{n+1}}$ and
        $\tilde{M}^{(\gamma)}_{n} := \prod_{j=0}^{n-1} \E{Y_{j}}$, 
        \begin{equation} \label{eq:rate-of-growth-mean-comparison}
            0 < \liminf_{N \to \infty} \frac{C_{N}}{\tilde{M}^{(\gamma)}_{N}} \leq \limsup_{N \to \infty} \frac{C_{N}}{\tilde{M}^{(\gamma)}_{N}} < \infty. 
        \end{equation}
        and
        \begin{equation} \label{eq:truncation-converges}
        \sum_{n=0}^{\infty} \frac{\E{\widehat{Y}^{(\gamma)}_{n} - 1}}{\tilde{M}^{(\gamma)}_{n+1}} < \infty,
        \end{equation}
    Moreover, the sequence $(C_{n})_{n \in \mathbb{N}_0}$ is a rate of growth if and only if, for some $\gamma > 0$, $\lim_{N \to \infty} \frac{C_{N}}{\tilde{M}^{(\gamma)}_{N}}$ exists in $(0,\infty)$, and Equations~\eqref{eq:truncation-determines-original} and~\eqref{eq:truncation-converges} are satisfied.   
\end{thm}

\begin{rmq}
    The criteria for a sequence $(C_{n})_{n \in \mathbb{N}_0}$ to be a rate of growth appears as~\cite[Lemma~1]{dsouza-rates-of-growth}, based on adapting the work of Barbour and Schuh~\cite{schuh-barbour-infinite-mean} for BGW processes with infinite mean. However, it is not immediately clear how the inverse generating functions used to define the Heyde martingale in~\cite{schuh-barbour-infinite-mean} can be adapted to the varying environments setting, and whilst it is possible this can be done, some details of this adaptation are not provided in detail in~\cite{dsouza-rates-of-growth}. The proof of Theorem~\ref{thm:rate-of-growth} uses generating functions only in the proof of the necessity of the lower bound in~\eqref{eq:rate-of-growth-mean-comparison}. To handle the varying-environment setting, we use a diagonalisation argument to construct the required Heyde martingale, thus avoiding inverses; see Proposition~\ref{prop:annoying-part}. It remains an open problem to prove the necessity of~\eqref{eq:rate-of-growth-mean-comparison} without using generating functions entirely.   {\small\ensymboldremark }
\end{rmq}

\begin{rmq}
Each of the criteria appearing in Theorem~\ref{thm:rate-of-growth} have natural interpretations. 
\begin{enumerate}
\item First, suppose that $(\tilde{Z}^{(\gamma)}_{n})_{n \in \mathbb{N}_0}$ denotes the BPVE with truncated laws $\boldsymbol{\mu}^{(\gamma)}$ such that each $\mu_{n}^{(\gamma)}$ has the law of $Y^{(\gamma)}_{n}$. By coupling this with the original process, it is evident that $(C_{n})_{n\in \mathbb{N}_0}$ is a growth scale of $(Z_{n})_{n \in \mathbb{N}_0}$ if and only if it is a growth scale of $(\tilde{Z}^{(\gamma)}_{n})_{n \in \mathbb{N}_0}$ for some $\gamma > 0$, and both processes coincide. By a Borel-Cantelli argument, is equivalent to~\eqref{eq:truncation-determines-original} being satisfied. 
\item Next, $(C_{n})_{n \in \mathbb{N}_0}$ is a growth scale of $\tilde{Z}^{(\gamma)}_{n})_{n \in \mathbb{N}_0}$ when $(\tilde{Z}^{(\gamma)}_{n}/C_{n})$ is bounded above and below, hence using a telescoping product
\begin{linenomath}
\begin{align*}
\frac{\tilde{Z}^{(\gamma)}_{n}}{C_{n}} = \prod_{j=0}^{n} \frac{\tilde{Z}^{(\gamma)}_{n}}{\tilde{Z}^{(\gamma)}_{n-1}} \times \frac{C_{n-1}}{C_{n}} & = \prod_{j=0}^{n} \frac{\sum_{j=1}^{\tilde{Z}^{(\gamma)}_{n-1}} X_{n,i} \mathbf{1}_{X_{n,i} \leq C_{n+1}}}{\tilde{Z}^{(\gamma)}_{n-1}} \times \frac{C_{n-1}}{C_{n}} 
\\ & \approx \prod_{j=0}^{n} \E{X_{n} \mathbf{1}_{X_{n} \leq C_{n+1}}} \times \frac{C_{n-1}}{C_{n}} = \frac{\tilde{M}^{(\gamma)}_{n}}{C_{n}}, 
\end{align*}
\end{linenomath}
where the approximation (at least morally, on $\tilde{Z}^{(\gamma)}_{n} \to \infty$) uses the strong law of large numbers. This implies that~\eqref{eq:rate-of-growth-mean-comparison} is satisfied. However, this logic is difficult to make precise, not least because it may be the case that $\tilde{Z}^{(\gamma)}_{n}$ is bounded. Making it rigorous requires usage of a Heyde martingale in Proposition~\ref{prop:annoying-part}.
\item The previous two steps now imply that, $(C_{n})_{n \in \mathbb{N}_0}$ is a growth scale of $(Z_{n})_{n \in \mathbb{N}_0}$ if and only if $(\tilde{M}^{(\gamma)}_{n})_{n \in \mathbb{N}_0}$ is a growth rate of $(\tilde{Z}^{(\gamma)}_{n})_{n \in \mathbb{N}_0}$. Thus, Equation~\eqref{eq:truncation-converges} follows from Theorem~\ref{thm:unif-integrable}. 
\end{enumerate}
In particular, the proof of Theorem~\ref{thm:rate-of-growth} shows that $(C_{n})_{n \in \mathbb{N}_0}$ is a growth scale for $(Z_{n})_{n \in \mathbb{N}_0}$ if and only if the sequence $(\tilde{M}^{(\gamma)}_{n})_{n \in \mathbb{N}_0}$  is a rate of growth of $(Z_{n})_{n \in \mathbb{N}_0}$. {\small\ensymboldremark }
\end{rmq}

\begin{rmq} \label{rmq:gamma-arbitrary}
    It turns out that Equations~\eqref{eq:truncation-determines-original},~\eqref{eq:rate-of-growth-mean-comparison} and~\eqref{eq:truncation-converges} are satisfied for some $\gamma > 0$ if and only if they are satisfied for any $\gamma >0$ such that $\E{Y^{(\gamma)}_{n}} > 0$ for all $n \in \mathbb{N}_0$. This is proved rigorously in Proposition~\ref{prop:gamma-arbitrary}.  {\small\ensymboldremark }
\end{rmq}

As a final application of Theorem~\ref{thm:rate-of-growth}, we provide a novel proof of Lindvall's theorem~\cite{lindval} on the convergence of BPVEs that does not require Church's theorem~\cite{church} (this was first adapted to the defective setting in~\cite{defective-bpves}):

\begin{cor}[{\cite[Theorem 1]{lindval}}] \label{cor:lindval}
Suppose $(Z_{n})_{n\in \mathbb{N}_0}$ is a BPVE with offspring sequence $\boldsymbol{\mu}$. Then, $\lim_{n \to \infty} Z_{n}$ exists almost surely. Moreover, 
\begin{equation} \label{eq:lindval}
    \Prob{\lim_{n\to \infty} Z_{n} = 0} + \Prob{\lim_{n\to \infty} Z_{n} = \infty} = 1 \quad \text{iff} \quad \sum_{n=0}^{\infty} \left(1 - \mu_{n}\left(\{1\}\right) \right) = \infty.   
\end{equation}
\end{cor}

\subsection{Organisation of the paper}

The paper is organised as follows. In Section~\ref{sec:unif-int} we prove the Kesten--Stigum type criterion, Theorem~\ref{thm:unif-integrable}. In Section~\ref{sec:growth-scale} we prove the growth-scale criterion, Theorem~\ref{thm:rate-of-growth}, together with the auxiliary results needed to justify the truncation and comparison steps. In Section~\ref{sec:cor} we prove Corollary~\ref{cor:lindval}. The proof of Corollary~\ref{cor:lindval} can be read independently of most of the preceding section.

\section{Proof of Theorem~\ref{thm:unif-integrable}} \label{sec:unif-int}
\subsection{Spine construction}
 It will be helpful to use \emph{Ulam-Harris} labelling of trees, where vertices in rooted trees are labelled by elements of $\mathcal{U}:= \bigcup_{n=0}^{\infty} \mathbb{N}^{n}$, with $\mathbb{N}^{0} := \left\{\varnothing\right\}$. We write elements $(u_1, \ldots, u_{m}) \in \mathcal{U}$ as $u_1 \cdots u_{m}$ and interpret $u = u_1 \cdots u_m$ as the $u_{m}$th \emph{child} of $u_1 \cdots u_{m-1}$. We use $|\cdot|$ to measure the \emph{length} of a tuple $u$, so that, if $u = \varnothing$ we set $|u| = 0$, whilst if $u = u_{1} \cdots u_{k}$ then $|u| = k$. Next, for $u \in \mathcal{U}$ and $k \in \mathbb{N}_0$, $k \leq |u|$, we write $u_{|k} : = u_1 \cdots u_{k}$ for the $k$th \emph{ancestor} of $u$. We say that $T \subseteq \mathcal{U}$ is a \emph{(sibling closed) tree}, if $\varnothing \in T$, and if $u = u_1 \cdots u_{m} \in T$, we have $u_{1} \cdots u_{k} \in T$ for all $k \leq m$, and $u_{1} \cdots u_{m-1} \ell \in T$ for all $\ell \leq u_{m}$. Given $u = u_1 \cdots u_m$ and $v = v_1 \cdots v_n\in \mathcal{U}$, we denote by $uv:= u_1 \cdots u_m v_1 \cdots v_n \in \mathcal{U}$. Here, we adopt the convention $\varnothing v = v = v \varnothing$. Then, for a tree $T$, given $u \in T$, we denote by $T^{\downarrow u}$ the subtree rooted at $u$, that is, the tree $\left\{v \in \mathcal{U}\colon uv \in T\right\}$. 

We identify trees as graphs in the natural way, and denote by $\mathscr{T}$ the space of locally finite trees. Given a tree $T \in \mathscr{T}$ and $n \in \mathbb{N}_{0}$, we denote by $T_{|n}:= \left\{v \in T\colon |v| \leq n\right\}$ the subtree obtained by retaining vertices of distance at most $n$ from the root. For $T \in \mathscr{T}$, we denote by $[T]_{n} := \left\{T' \in \mathscr{T}\colon T'_{|n} = T_{|n}\right\}$. Note that if one views $\mathscr{T}$ as a topological space with the local topology - so that sets of the form $[T]_{n}$ are open - then $\mathscr{T}$ is Polish, for example, metrisable with the metric $d$ such that $d(T,T') := \left(1 + \sup\{n \in \mathbb{N}\colon T_{|n} = T'_{|n} \}\right)^{-1}$.
We denote by $Z_{n}\colon \mathscr{T} \rightarrow \mathbb{N}_0$ the function describing the size of the $n$th generation of the tree, so that
\[Z_{n}(T):= |\left\{u \in T\colon |u|=n \right\}|.\]

We denote by $\mathscr{T}_{\spine}$ the space of infinite locally finite trees with a distinguished infinite path belonging to the tree, i.e. 
\[
\mathscr{T}_{\spine} := \left\{(T, \boldsymbol{v})\colon  T \in \mathscr{T}, \boldsymbol{v} = (v_1, v_2, \ldots) \in \mathbb{N}^{\mathbb{N}} \text{ and } \, \forall i \in \mathbb{N} \; v_1\cdots v_i \in T \right\}
\] 
For $(T, \boldsymbol{v}) \in \mathscr{T}_{\spine}$, and $i \in \mathbb{N}$ we call $v_1 \cdots v_{i}$ the $i$th \emph{spine node}, with $\varnothing$ the $0$th \emph{spine node}. 
Finally, for $n \in \mathbb{N}$ we denote by $[T,\boldsymbol{v}]_{n} := \left\{(T', \boldsymbol{v}') \in \mathscr{T}_{\spine}\colon T' \in [T]_{n}, \; v'_1 \cdots v'_n = v_1 \cdots v_n\right\}.$ Note that $\mathscr{T}_{\spine}$ is also Polish with the topology making each $[T,\boldsymbol{v}]_{n}$ open. We view $\mathscr{T}, \mathscr{T}_{\spine}$ as measure spaces with their Borel $\sigma$-algebras $\mathscr{F}, \mathscr{S}$ respectively. For each $n \in \mathbb{N}$ we denote by $\mathscr{F}_{n} := \sigma\left([T]_{n}\colon T \in \mathscr{T}\right)$ and $\widehat{\mathscr{F}}_{n} := \sigma\left([T, \boldsymbol{v}]_{n}\colon T \in \mathscr{T}_{\spine}\right)$.  

Note that by adding children to the tree in the natural way, for a given sequence of laws $\boldsymbol{\mu}$, we may view a BPVE as a $\mathscr{T}$ valued random variable. We denote such a random variable by $\mathcal{T}$ (excluding the dependence on $\mu$ for simplicity of notation) and denote by $\Prob[\boldsymbol{\mu}]{\cdot}$ the pullback measure on $\mathscr{T}$ associated with $\mathcal{T}$. Suppose that $\mathcal{S}$ denotes the shift operator on sequences, so that if $\boldsymbol{\mu} = (\mu_0, \mu_1, \ldots)$, $\mathcal{S}\boldsymbol{\mu} = (\mu_1, \mu_2, \ldots)$. We denote by $\mathcal{S}^{n}$ the $n$-fold composition of shifts, i.e. $\mathcal{S}^n\boldsymbol{\mu} = (\mu_{n}, \mu_{n+1}, \ldots)$. 

We now construct a branching process in varying environments with a spine, in a manner analogous to previous constructions in, e.g.~\cite{yaglom-limit-law}. In particular, we give the root $\widehat{X}_0$ children, where $\widehat{X}_0$ has the size-biased distribution of $X_0$ (a random variable with law $\mu_0$). We sample a child uniformly at random to continue the spine, and attach to each non-spine child independently a sub-tree sampled from $\Prob[\mathcal{S}\boldsymbol{\mu}]{\cdot}$. The spine construction proceeds iteratively: if the current spine individual is in generation $n$, then it produces $\widehat{X}_n$ children, where $\widehat{X}_n$ has the size-biased distribution of $X_n$ with law $\mu_n$. Again, one of these children is chosen uniformly to extend the spine, while every other child independently starts an attached subtree with law $\Prob[\mathcal{S}^{n+1}\boldsymbol{\mu}]{\cdot}$. The spine construction leads to a random tree $\widehat{T}$ taking values in $\mathscr{T}_{\spine}$. For a given sequence of laws $\boldsymbol{\mu}$, we denote by $\Probtwo[\boldsymbol{\mu}]{\cdot}$ the pullback measure associated with this $\mathscr{T}_{\spine}$-valued random variable. Abusing notation, we also denote by $\Probtwo[\boldsymbol{\mu}]{\cdot}$ the induced measure on $\mathscr{T}$, obtained by integrating over possible distinguished paths $\boldsymbol{v}$. Finally, recalling that $\mathscr{F}_{n}$, denotes the sub-$\sigma$-algebra $\sigma([T]_{n}, T \in \mathscr{T})$,  corresponding to information about the first $n$ generations of a tree, we denote by $\mathbb{P}_{\boldsymbol{\mu} \restriction \mathscr{F}_n}$ and $\mathbb{Q}_{\boldsymbol{\mu} \restriction \mathscr{F}_n}$ the restrictions of the measures $\mathbb{P}_{\boldsymbol{\mu}}, \mathbb{Q}_{\boldsymbol{\mu}}$ to $\mathscr{F}_{n}$. 

\subsection{Proof of Item~\ref{item:non-deg-limit} of Theorem~\ref{thm:unif-integrable}}
In this section we use a spine argument to prove Item~\ref{item:non-deg-limit} of Theorem~\ref{thm:unif-integrable}. We note that this result follows a straightforward extension of the approach of~\cite{conceptual-xlogx} to the varying environments case, however, we include the proof for completeness. First, we have the following two lemmas:

\begin{lem} \label{lem:change-of-measure}
    As measures on $\mathscr{T}$, for each $n \in \mathbb{N}_0$ we have 
\begin{equation} \label{eq:change-of-measure}
\mathbb{Q}_{\boldsymbol{\mu} \restriction \mathcal{F}_n} \ll \mathbb{P}_{\boldsymbol{\mu} \restriction \mathcal{F}_n}, \text{ with } \frac{\dd \mathbb{Q}_{\boldsymbol{\mu} \restriction \mathcal{F}_n}}{\dd \mathbb{P}_{\boldsymbol{\mu} \restriction \mathcal{F}_n}} = \frac{Z_{n}}{M_n}.
\end{equation}
In particular, if we write $W :=  \limsup_{n \to \infty} \frac{Z_{n}}{M_n}$, for any $S \in \mathscr{F}$ we have  
\begin{equation} \label{eq:limiting-decomposition}
\mathbb{Q}_{\boldsymbol{\mu}}(S) = \int_{S} W \mathbb{P}_{\boldsymbol{\mu}} + \mathbb{Q}_{\boldsymbol{\mu}}\left(S \cap \{W = \infty\}\right). 
\end{equation}
\end{lem}
\begin{proof}
Suppose that $(T, \boldsymbol{v}) \in \mathscr{T}_{\spine}$, and denote by $\mathcal{S}\boldsymbol{v}$ the sequence $(v_2, v_3, \ldots)$. First note that, for any such $(T, \boldsymbol{v})$, if $Z^{T}_{1} = J \in \mathbb{N}$
\[
\Probtwo[\boldsymbol{\mu}]{[T, \boldsymbol{v}]_1} = \Prob{\widehat{X}_0 = J} \times \Prob{v_1 \text{ selected to continue spine}} = \frac{1}{J} \times J \times\frac{\Prob{X_0 = J}}{\E{X_0}} = \frac{\Prob[\boldsymbol{\mu}]{[T]_1}}{M_1}.
\]
Then, by induction, and the conditional independence of subtrees associated with the process, again assuming that $Z^{T}_{1} = J \in \mathbb{N}$, note that we have
\begin{linenomath}
\begin{align}
\nonumber \Probtwo[\boldsymbol{\mu}]{[T, \boldsymbol{v}]_{n}} & = \Prob{\widehat{X}_{0} = J} \times \frac{1}{J} \times \Probtwo[\mathcal{S}\boldsymbol{\mu}]{[T^{\downarrow v_1}, \mathcal{S}\boldsymbol{v}]_{n-1}} \times \prod_{r \neq v_{1}} \Prob[\mathcal{S}\boldsymbol{\mu}]{[T^{\downarrow r}]_{n-1}} 
\\ & = \frac{\Prob{X_0 = J} \Prob[\mathcal{S}\boldsymbol{\mu}]{[T^{\downarrow v_1}]_{n-1}}}{M_n} \times \prod_{r \neq v_{1}} \Prob[\mathcal{S}\boldsymbol{\mu}]{[T^{\downarrow r}]_{n-1}} = \frac{\Prob[\boldsymbol{\mu}]{[T]_{n}}}{M_n}
\end{align}
\end{linenomath}
We deduce that $\Probtwo[\boldsymbol{\mu}]{[T]_n} = \frac{Z^{T}_{n}}{M_n} \Prob[\boldsymbol{\mu}]{[T]_n}$ 
so that~\eqref{eq:change-of-measure} follows. Finally,~\eqref{eq:limiting-decomposition} follows from~\eqref{eq:change-of-measure} and a known measure decomposition result (see e.g.~\cite[Equation~(12.4)]{lyons}).
\end{proof}

\begin{lem}[{\cite[Theorem~6.2]{petrov1995limit}}] \label{lem:three-series-positive}
    Suppose that $(S_{j})_{j \in \mathbb{N}}$ is a sequence of mutually independent random variables taking values in $[0, \infty)$. Let $K > 0$ be given. Then $\sum_{j=1}^{\infty} S_{j} < \infty$ almost surely if and only if 
    \begin{equation} \label{eq:kolmogorov-three-series-positive}
        \sum_{j=1}^{\infty} \Prob{S_{j} > K} < \infty \quad \text{ and } \quad \sum_{j=1}^{\infty} \E{S_{j} \mathbf{1}_{S_{j} \leq K}} < \infty. 
    \end{equation}
    \hfill $\blacksquare$
\end{lem}

Since the proof of Theorem~\ref{thm:unif-integrable} uses the spine process $(\widehat{Z}_{n})_{n \in \mathbb{N}_0}$, we introduce some notation, and compute some relevant quantities related to this process. First, suppose that we denote by $(N_{i})_{i = 0, \ldots, n-1}$ the number of non-spine children of the $i$th spine node in $(\widehat{Z}_{n})_{n \in \mathbb{N}}$ (where we recall that the root is the $0$th spine node). Denote by $(\widehat{Z}^{(\ell)}_{n, i})_{\ell= 1, \ldots, N_{i}}$ the contribution descendents of each of the $N_{i}$ children make to $\widehat{Z}_{n}$. Since the spine always contributes one, we have the decomposition
\begin{equation} \label{eq:spine-sum-decomp}
    \widehat{Z}_{n} = 1 +  \sum_{i=0}^{n-1} \sum_{\ell = 1}^{N_{i}} \widehat{Z}^{(\ell)}_{n,i}. 
*\end{equation}
By the dynamics of $(\widehat{Z}_{n})_{n \in \mathbb{N}}$, for $n > i$, each $\widehat{Z}^{(\ell)}_{n, i}$ is distributed like $\tilde{Z}^{(i)}_{n-(i+1)}$, where $(\tilde{Z}^{(i)}_{n})_{n \in \mathbb{N}}$ is a BPVE with offspring sequence $\mathcal{S}^{i+1}\boldsymbol{\mu}$. Moreover, since the $i$th spine node has $\widehat{X}_{i}$ offspring, one of which is the next node on the spine, by independence of $N_{i}$ and the processes $\widehat{Z}^{(\ell)}_{n,i}$ we have 
\[
\E{\sum_{\ell = 1}^{N_{i}} \widehat{Z}^{(\ell)}_{n,i} \, \bigg | \, N_{i}} = N_i \E{\widehat{Z}^{(\ell)}_{n,i}} =  \left(\widehat{X}_{i} - 1 \right) \E{\widehat{Z}^{(\ell)}_{n,i}} = \left(\widehat{X}_{i} - 1\right) \left(\prod_{j=i+1}^{n-1} m_{j} \right).
\] 
Therefore, if $\widehat{Z}_{n,i} := \sum_{\ell = 1}^{N_{i}} \widehat{Z}^{(\ell)}_{n, i}$, we have 
\begin{equation} \label{eq:martingale-property-requirement}
\frac{\E{\widehat{Z}_{n,i} \, \big | \, N_{i}}}{M_n} = \frac{N_{i}}{M_{i+1}} = \frac{\widehat{X}_i - 1}{M_{i+1}}. 
\end{equation}

\begin{proof}[Proof of Item~\ref{item:non-deg-limit} of Theorem~\ref{thm:unif-integrable}]
Following the notation defined above, note that each of the values $(N_{i})_{i \in \mathbb{N}_0}$ are independent, and that each $\widehat{Z}^{(\ell)}_{n, i}$ is independent of $N_i$. Thus, if we denote by $\mathscr{Y} := \sigma(N_{i};{i \in \mathbb{N}_0})$ the $\sigma$-algebra generated by the values of $(N_{i})_{i \in \mathbb{N}_0}$, we have
\begin{linenomath}
    \begin{align*}
        \E{\frac{\widehat{Z}_{n} - 1}{M_{n}} \, \bigg | \, \widehat{\mathscr{F}}_{n-1},\mathscr{Y}} = 
        \frac{N_{n-1}}{M_{n}} + \sum_{i=0}^{n-2} \sum_{\ell = 1}^{N_{i}} \E{\frac{\widehat{Z}^{(\ell)}_{n, i}}{M_n} \, \bigg | \, \widehat{\mathscr{F}}_{n-1}, \mathscr{Y}} \geq \sum_{i=0}^{n-2} \sum_{\ell = 1}^{N_{i}} \frac{\widehat{Z}^{(\ell)}_{n-1, i}}{M_{n-1}} = \frac{\widehat{Z}_{n-1} - 1}{M_{n-1}}.
    \end{align*}
\end{linenomath}
Thus, conditional on $\mathscr{Y}$, $\left(\left(\widehat{Z}_{n} - 1\right)/M_{n}\right)_{n \in \mathbb{N}_0}$ is a sub-martingale. Note also that by~\eqref{eq:martingale-property-requirement}
\begin{equation} \label{eq:expected-value-infinite-sum}
\E{\frac{\widehat{Z}_{n} - 1}{M_{n}}\, \bigg | \, \mathscr{Y}} \leq \sum_{i=0}^{\infty} \frac{N_{i}}{M_{i+1}} < \infty, \quad {\text{almost surely,}} 
\end{equation}
where the last step follows from Lemma~\eqref{lem:three-series-positive}, the fact that $N_{i} = \widehat{X}_{i}-1$ and the assumption that~\eqref{eq:l1-requirement} is satisfied. Thus, by the martingale convergence theorem, $\lim_{n \to \infty} \left(\widehat{Z}_{n} - 1\right)/M_{n}$ exists and is finite almost surely. Since in addition, $\limsup_{n \to \infty} 1/M_{n} = 0 $, we deduce that  
\begin{equation} \label{eq:limit-exists}
\Probtwo[\boldsymbol{\mu}]{\left\{T \in \mathscr{T}\colon \lim_{n \to \infty} \frac{Z_{n}}{M_{n}} \; \text{exists and is finite}\right\}} = 1,
\end{equation}
which implies by Equation~\eqref{eq:limiting-decomposition} that $\E{W} = \E[\mathbb{P}]{\lim_{n \to \infty} Z_{n}/M_{n}} = \Probtwo[\boldsymbol{\mu}]{\mathscr{T}} = 1$. Thus, $L^1$ convergence follows from the almost sure convergence of $Z_{n}/M_{n}$ and Scheff\'e's lemma. Moreover, by~\eqref{eq:limiting-decomposition}
\begin{linenomath}
    \begin{align} \label{eq:w-is-zero}
        \Probtwo[\boldsymbol{\mu}]{\left\{W = 0\right\}} = \E[\mathbb{P}_{\boldsymbol{\mu}}]{W\mathbf{1}_{W = 0}} + \Probtwo[\boldsymbol{\mu}]{\{W = 0\} \cap \{W = \infty\}} = 0.   
    \end{align}
\end{linenomath}
Thus,
\begin{linenomath}
    \begin{align}
\Prob{W > 0} = \E[\mathbb{P}_{\boldsymbol{\mu}}]{\frac{1}{W} \times W \mathbf{1}_{W >0}} \stackrel{\eqref{eq:limiting-decomposition}}{=} \E[\mathbb{Q}_{\boldsymbol{\mu}}]{\frac{1}{W} \mathbf{1}_{W >0}} \stackrel{\eqref{eq:w-is-zero}}{=} \E[\mathbb{Q}_{\boldsymbol{\mu}}]{\frac{1}{W}}  = \E{\lim_{n \to \infty} \frac{M_{n}}{\widehat{Z}_{n}}}. 
    \end{align}
\end{linenomath}
Now, by Jensen's inequality and Fatou's lemma, 
\begin{linenomath}
 \[\E{\lim_{n \to \infty} \frac{M_{n}}{\widehat{Z}_{n}}} \geq \E{\frac{1}{\E{\lim_{n \to \infty} \widehat{Z}_{n}/M_{n} \, | \, \mathscr{Y}}}} \geq \E{\frac{1}{\liminf_{n \to \infty}\E{ \widehat{Z}_{n}/M_{n} \, | \, \mathscr{Y}}}} \geq \E{\frac{1}{\sum_{i=0}^{\infty} \frac{N_i}{M_{i+1}}}},\]
\end{linenomath}
which implies~\eqref{eq:prob-pos}. 
\end{proof}

\subsection{Proof of Item~\ref{item:deg-limit} of Theorem~\ref{thm:unif-integrable}}
The main ingredient in the proof of Item~\ref{item:deg-limit} of Theorem~\ref{thm:unif-integrable}, that differs from the classical spine argument is the following lemma:
\begin{lem} \label{lem:trunc}
Suppose that $(Z_{n})_{n \in \mathbb{N}_0}$ is a BPVE with offspring random variables $(X_{n})_{n \in \mathbb{N}_0}$ such that, for some $C >0$, for all $i \in \mathbb{N}_0$ we have $\widehat{X}_{i} \leq C M_{i+1}$ almost surely. Moreover, suppose that 
\[
\sum_{i=0}^{\infty} \frac{\E{\widehat{X}_{i} - 1}}{M_{i+1}} = \infty. 
\]
Then, $\lim_{n \to \infty} Z_n/M_n = 0$ almost surely. 
\end{lem}

We first use Lemma~\ref{lem:trunc} to complete the proof of Theorem~\ref{thm:unif-integrable}, then prove Lemma~\ref{lem:trunc} in the following section. 

\begin{proof}[Proof of Item~\ref{item:deg-limit} of Theorem~\ref{thm:unif-integrable}]
First suppose that $\limsup_{n \to \infty} 1 / M_{n} = \infty$. Then, we may find a subsequence $(M_{n_{i}})_{i \in \mathbb{N}}$ such that $\sum_{i=1}^{\infty} M_{n_i} < \infty$. By Markov inequality 
\[
\Prob{Z_{n_i} > 0} \leq M_{n_i},
\]
thus, by the Borel-Cantelli lemma, $\Prob{Z_{n_i} > 0 \text{ for infinitely many $i$}} = 0$. This implies \[\lim_{n \to \infty} Z_{n} = 0 \quad  \text{almost surely}.\] 
Next, we consider the case that the sum in~\eqref{eq:l1-requirement} diverges, because for all $K' \in (0, \infty)$
\begin{equation} \label{eq:gen-not-xlogx}
\sum_{n=0}^{\infty} \Prob{\widehat{X}_{n} - 1\geq K' M_{n+1}} = \infty. 
\end{equation}
Following the notation regarding the spine process in the previous section, note that $\widehat{Z}_{n} > \widehat{Z}_{n,n-1} = N_{n-1}$. Thus, for any $K \in (0,\infty)$
\begin{linenomath}
\begin{align*}
    \Prob{\limsup_{n \to \infty} \frac{\widehat{Z}_n}{M_{n}} \geq K } \geq \Prob{\limsup_{n \to \infty} \frac{N_{n-1}}{M_{n}} \geq K'} = \Prob{\limsup_{n \to \infty} \frac{\widehat{X}_{n-1} - 1}{M_{n}} \geq K'} \stackrel{\eqref{eq:gen-not-xlogx}}{=} 1,
 \end{align*}
\end{linenomath}
where the last equality follows from converse of the Borel-Cantelli lemma. Since $K'$ was arbitrary, we deduce that $\limsup_{n \to \infty} \widehat{Z}_n/{M_{n}} = \infty$ almost surely, so that 
\[
\Probtwo[\boldsymbol{\mu}]{\left\{T \in \mathscr{T}\colon \limsup_{n \to \infty} \frac{Z_{n}}{M_{n}} = \infty\right\}} = 1. 
\]
Equation~\eqref{eq:limiting-decomposition} therefore implies $\Prob{\lim_{n \to \infty} \frac{Z_{n}}{M_{n}} = 0} = 1$ as required.

Otherwise, for some $K' \geq K$
\begin{align}  \label{eq:gen-not-xlogx-two} 
\sum_{n=0}^{\infty} \Prob{\widehat{X}_{n} - 1\geq K' M_{n+1}} < \infty, \quad \text{yet} \quad \sum_{n=0}^{\infty} \frac{\E{\left(\widehat{X}_{n} - 1 \right)\mathbf{1}_{\widehat{X}_n -1 \leq K' M_{n+1}}}}{M_{n+1}} = \infty. 
\end{align}
Note that the terms of the first sum in~\eqref{eq:gen-not-xlogx-two} are decreasing in $K'$ whilst the terms in the second are increasing in $K'$. Hence~\eqref{eq:gen-not-xlogx-two} is satisfied for all $K' \geq K$. Now, suppose that $(Z_{n}^{(K')})_{n \in \mathbb{N}_0}$ denotes a truncation of the original process $(Z_{n})_{n \in \mathbb{N}_0}$, with offspring laws having the distributions of the random variables $(X_{n}\mathbf{1}_{X_{n} \leq K' M_{n+1} + 1})_{n \in \mathbb{N}_0}$. Suppose that $(M^{(K')}_{n})_{n \in \mathbb{N}}$ denote the means associated with the truncated process. Then, using~\eqref{eq:gen-not-xlogx-two}, by making $K'$ larger if necessary, we may assume that $\sum_{n=0}^{\infty} \Prob{\widehat{X}_{n} - 1\geq K' M_{n+1}} < 1$, hence
\begin{linenomath}
\begin{align} \label{eq:inf-mean-ratio}
 \inf_{n \in \mathbb{N}} \left\{\frac{M_{n}^{(K')}}{M_n}\right\} & = \inf_{n \in \mathbb{N}} \left\{\prod_{j=0}^{n-1} \frac{\E{X_{j} \mathbf{1}_{X_{j} \leq K' M_{j+1}+1}}}{m_j} \right\}  = \inf_{n \in \mathbb{N}} \left\{ \prod_{j=1}^{n-1} \left(1 - \frac{\E{X_{j} \mathbf{1}_{X_j > K' M_{j+1} +1}}}{m_j}\right)\right\}
\\ \nonumber & = \inf_{n \in \mathbb{N}} \left\{ \prod_{j=0}^{n-1} \left(1 - \Prob{\widehat{X}_{j} > K' M_{j+1} + 1}\right) \right\} \geq 1 - \sum_{j=0}^{\infty} \Prob{\widehat{X}_{j} - 1\geq K' M_{j+1}} > 0.  
\end{align}
\end{linenomath}
Now, note that, since $\inf_{n \in \mathbb{N}_0} M_{n} > 0$, and the inverse of the infimum in Equation~\eqref{eq:inf-mean-ratio} is bounded from above, we can find a bound $C > 0$ such that
\[
X_{n} \mathbf{1}_{X_{n} \leq K' M_{n+1} + 1} \leq C M^{(K')}_{n+1}.  
\]
Thus, by Lemma~\ref{lem:trunc}, for all $K' \geq K$ we have $\lim_{n \to \infty} \frac{Z^{(K')}_n}{M^{(K')}_n} = 0$. 
On the other hand, we may couple $(Z_{n}^{(K')})_{n \in \mathbb{N}_0}$ and $(Z_{n})_{n \in \mathbb{N}_0}$ so that $(Z_{n}^{(K)})_{n \in \mathbb{N}_0}$ is formed from $(Z_{n})_{n \in \mathbb{N}_0}$ by removing all offspring of individuals in the $n$th generation that have more that $K'M_{n+1} + 1$ offspring. But on this coupling, by Doob's martingale inequality, 
\[
\Prob{\exists n \in \mathbb{N}\colon \; Z_{n} \neq Z^{(K')}_{n} } \leq \Prob{\exists n \in \mathbb{N}\colon \frac{Z_n}{M_{n+1}} > K'} = \Prob{\sup_{n \in \mathbb{N}_0} \frac{Z_{n}}{M_n} > K'} \leq \frac{1}{K'}
\]
Thus, 
\[
\Prob{\lim_{n \to \infty} \frac{Z_{n}}{M_{n}} = 0} \geq \Prob{\left\{ \lim_{n \to \infty} \frac{Z^{(K')}_n}{M^{(K')}_{n}} = 0 \right\} \cap \left\{\forall n \in \mathbb{N} \; Z_{n} = Z^{(K')}_{n}\right\}} \geq 1 - \frac{1}{K'}.
\]
Since $K'$ can be made arbitrarily large, the proof of Theorem~\ref{thm:unif-integrable} follows. 
\end{proof}

\subsection{Proof of Lemma~\ref{lem:trunc}}
We also start this section with two lemmas. Lemma~\ref{lem:variance} first appeared in~\cite{fearn}. 

\begin{lem} \label{lem:variance}
    If $(Z_{n})_{n \in \mathbb{N}_0}$ is a BPVE with offspring sequence $\boldsymbol{\mu}$ such that for each $X_{n}$, $\E{X_n} < \infty$, we have 
    \begin{equation} \label{eq:variance}
    \Var{\frac{Z_{n}}{M_n}} = \sum_{j=0}^{n-1} \frac{\Var{X_j}}{M_{j}m_j^2} =
    \left(\sum_{j=0}^{n-2} \frac{\E{\widehat{X}_j - 1}}{M_{j+1}}\right) + \frac{\E{\widehat{X}_{n-1}}}{M_{n}} - 1. 
    \end{equation}    
\end{lem}
\begin{proof}
Note that for a random number $N$ of independent random variables, by the law of total variance we have 
\begin{equation} \label{eq:variance-sum-formula}
    \Var{\sum_{j=1}^{N} Y_{i}} = \Var{Y_1}\E{N} + \Var{N} \E{Y_{1}}^2.
\end{equation}
By~\eqref{eq:variance-sum-formula}, we have
\[
\Var{\frac{Z_{n}}{M_n}} = \frac{M_{n-1}\Var{X_{n-1}} + \Var{Z_{n-1}} m_{n-1}^2}{M_n^2} = \frac{\Var{X_{n-1}}}{M_{n-1} m_{n-1}^2} + \Var{\frac{Z_{n-1}}{M_{n-1}}},  
\]
and thus the first equality in~\eqref{eq:variance} follows from iteration and the fact $Z_0 = 1$. Re-writing the first equality in~\eqref{eq:variance}, we have 
\begin{linenomath}
    \begin{align*}
        \Var{\frac{Z_{n}}{M_n}} = \sum_{j=0}^{n-1} \frac{\E{X_{j}^2} - m_j^2}{M_{j}m_j^2} = \sum_{j=0}^{n-1} \frac{\E{\widehat{X}_j}}{M_{j+1}} - \sum_{j=0}^{n-1} \frac{1}{M_j} = \left(\sum_{j=0}^{n-2} \frac{\E{\widehat{X}_{j} - 1}}{M_{j+1}}\right) + \frac{\E{\widehat{X}_{n-1}}}{M_{n}} - 1. 
    \end{align*}
\end{linenomath}
\end{proof}

\begin{lem} \label{lem:summands-diverge}
    Suppose that $(U_{n,i})_{n\in \mathbb{N}_0, i < n}$ is an array of non-negative, independent random variables and let $\boldsymbol{u} = (u_{i})_{i \in \mathbb{N}_0} \in (0, \infty)^{\mathbb{N}_0}$ be a sequence with $\sum_{i=0}^{\infty} u_i = \infty$. Suppose that for some absolute constant $C > 1$, we have $u_{i} \leq C$ and for all $n \in \mathbb{N}$ we have
    \begin{equation} \label{eq:variance-mean-bound}
        \E{U_{n,i}} = u_i \quad \text{ and } \quad \Var{U_{n,i}} \leq Cu_{i}^2 +  \left(C + \sum_{j=i+1}^{n-2} u_{j}\right)u_{i}. 
    \end{equation}
    Then, for any $r \in \mathbb{N}_0, M > C$, $\eps > 0$ there exists $n =  n(\eps, M, r) > r \in \mathbb{N}$ such that 
    \begin{equation} \label{eq:large-prob-to-be-large}
    \Prob{\sum_{i=r}^{n-1} U_{n,i} \leq M/2} < \eps. 
    \end{equation}
\end{lem}
As the proof of Lemma~\ref{lem:summands-diverge} is a bit more computationally involved, we first prove Lemma~\ref{lem:trunc} and delay the proof of Lemma~\ref{lem:summands-diverge} to the end of the section. 

\begin{proof}[Proof of Lemma~\ref{lem:trunc}]
    By Lemma~\ref{lem:change-of-measure}, it suffices to show that, if $(\widehat{Z}_{n})_{n \in \mathbb{N}_0}$ denotes the spine process associated with $(Z_{n})_{n \in \mathbb{N}_0}$ then we have
    \begin{equation} \label{eq:need-to-prove}
    \limsup_{n \to \infty} \frac{\widehat{Z}_n}{M_n} = \infty. 
    \end{equation}
    But then, note that by~\eqref{eq:spine-sum-decomp} we can write 
    \begin{equation} \label{eq:spine-decomp-two}
    \frac{\widehat{Z}_n}{M_n} = \frac{1}{M_n} + \sum_{i=0}^{n-1}\frac{\widehat{Z}_{n,i}}{M_n}.
    \end{equation}
    Moreover, by Lemma~\ref{lem:variance} and independence of the summands involved,
\begin{linenomath}
\begin{align} \label{eq:variance-non-spine-descendents}
\nonumber \Var{\frac{\widehat{Z}_{n, i}}{M_{n}}} & = \frac{1}{M_{i+1}^2} \Var{\frac{\sum_{\ell=1}^{N_{i}}Z^{(\ell)}_{n,i}}{\prod_{j=i+1}^{n-1} m_{j}}} \stackrel{\eqref{eq:variance-sum-formula}}{=} \frac{\E{N_{i}}}{M_{i+1}^2} \Var{\frac{\widehat{Z}^{(1)}_{n,i}}{\prod_{j=i+1}^{n-1} m_{j}}} + \frac{\Var{N_{i}}}{M_{i+1}^2} \times 1
\\ & \nonumber \stackrel{\eqref{eq:variance}}{=} \frac{\E{\widehat{X}_i - 1}}{M_{i+1}^2} \left(\left(\sum_{j=i+1}^{n-2} \frac{\E{\widehat{X}_{j} - 1}}{\prod_{\ell=i+1}^{j} m_{\ell}}\right) + \frac{\E{\widehat{X}_{n-1}}}{\prod_{\ell=i+1}^{n-1} m_{\ell}} - 1 \right) + \frac{\Var{\widehat{X}_i - 1}}{M^2_{i+1}} 
\\ & \hspace{1.05mm} = \frac{\E{\widehat{X}_i - 1}}{M_{i+1}} \left(\left(\sum_{j=i+1}^{n-2} \frac{\E{\widehat{X}_{j} - 1}}{M_{j+1}}\right) + \frac{\E{\widehat{X}_{n-1}}}{M_{n}} - 1 \right) + \frac{\Var{\widehat{X}_i - 1}}{M^2_{i+1}}
\end{align}
\end{linenomath}
Now, set
\[
u_i:=\frac{\E{\widehat X_i-1}}{M_{i+1}},
\qquad
U_{n,i}:=\frac{\widehat Z_{n,i}}{M_n}.
\]
The assumptions imply that $\widehat X_i/M_{i+1}\leq C$ almost surely, and hence $u_i\leq C$. Moreover, by~\eqref{eq:variance-non-spine-descendents}, and increasing the constant if necessary, the variables $U_{n,i}$ satisfy the variance bound in Lemma~\ref{lem:summands-diverge}. The terminal case $i=n-1$ is covered separately, since then $\widehat Z_{n,n-1}=N_{n-1}=\widehat X_{n-1}-1$, and the same bound follows from $\widehat X_{n-1}/M_n\leq C$.

Applying Lemma~\ref{lem:summands-diverge}, for every $r\in\mathbb{N}_0$, every $M>C$ and every $\varepsilon>0$, there exists $n>r$ such that
\[
\Prob{\sum_{i=r}^{n-1}U_{n,i}\leq M/2}<\eps,
\]
hence, using the almost sure bound $\frac{\widehat Z_n}{M_n}
\geq
\sum_{i=r}^{n-1}U_{n,i}$, we have, 
    \[
    \Prob{\limsup_{n \to \infty} \frac{\widehat{Z}_{n}}{M_n} \leq M/2} \stackrel{\eqref{eq:spine-decomp-two}}{\leq} \Prob{\limsup_{n \to \infty} \sum_{i=0}^{n-1} U_{n,i} \leq M/2}  \stackrel{\eqref{eq:large-prob-to-be-large}}{=} 0.
    \]
    Equation~\eqref{eq:need-to-prove} follows. 
\end{proof}
We need only prove Lemma~\ref{lem:summands-diverge}.
\begin{proof}[Proof of Lemma~\ref{lem:summands-diverge}]
It suffices to prove the result for $r=0$. Indeed, for fixed $r\in\mathbb{N}_0$, apply the case $r=0$ to the reindexed tail array $(U_{n,i+r})_{n>r,\,i<n-r}$ and the sequence $(u_{i+r})_{i\in \mathbb{N}_0}$. The assumptions are inherited from the original array, and $\sum_{i=r}^{\infty}u_i=\infty$.
We start with the following claim:
\begin{clm} \label{clm:div-sum-ineq}
    For all $M>C$, for any $k \in \mathbb{N}$, there exists an increasing sequence $(n_{i})_{i \in \{0, \ldots, k\}}$ of integers, with $n_{0} = -1$, such that, with
    \[
    s_{j} := 
    \sum_{\ell = n_{j-1}+1}^{n_{j}} u_{\ell}
    \]
    for all $j \in \left\{1, \ldots, k \right\}$ we have 
    \begin{equation} \label{eq:means-dom}
     M \leq s_{j} \quad \text{ and } \quad C + \sum_{\ell = j+1}^{k} s_{\ell} \leq s_{j}.
    \end{equation} 
\end{clm}
Given the truth of the claim, for $\eps > 0$, fix $k$ sufficiently large that $\left(\frac{4(C+2)}{4(C+2)+1}\right)^k \leq \eps$ and choose a sequence $(n_{i})_{i = 0, \ldots k}$ such that~\eqref{eq:means-dom} is satisfied. Then, if we set $N -1 = n_{k}$,
\[
S_{n, j} :=
\sum_{\ell=n_{j-1}+1}^{n_{j}} U_{n, \ell} 
\]
note that $\E{S_{n,j}} = s_{j}$ for $j \in \left\{ 1, \ldots, k\right\}$. Moreover, since each of the summands $U_{n, \ell}$ are independent, by~\eqref{eq:variance-mean-bound} and the fact $n -2 = n_{k}-1$
\begin{align*}
\Var{S_{n,j}} & \leq C\left(\sum_{i=n_{j-1}+1}^{n_{j}} u_{i}^2\right) + \sum_{i=n_{j-1}+1}^{n_{j}} \left(C + \sum_{\ell=i+1}^{n_{k}-1} u_{\ell}\right)u_{i}
\\ & \leq  C\left(\sum_{i=n_{j-1}+1}^{n_{j}} u_{i}^2\right) + \sum_{i=n_{j-1}+1}^{n_{j}} \left(C + \sum_{\ell = j+1}^{k} s_{\ell} + \sum_{\ell=i+1}^{n_{j}} u_{\ell}\right)u_{i}
\\ & \stackrel{\eqref{eq:means-dom}}{\leq} C\left(\sum_{i=n_{j-1}+1}^{n_{j}} u_{i}^2\right) + \sum_{i=n_{j-1}+1}^{n_{j}} \left(s_{j} + \sum_{\ell=i+1}^{n_{j}} u_{\ell}\right)u_{i} 
\\ & \leq C\left(\sum_{i=n_{j-1}+1}^{n_{j}} u_{i}^2\right) + 2s_{j}^2 
\leq (C+2) s_{j}^2. 
 \end{align*}
By  Cantelli's inequality, 
\[
\Prob{S_{n,j} \leq M/2} \stackrel{\eqref{eq:means-dom}}{\leq} \Prob{S_{n,j} \leq \frac{s_{j}}{2}} = \Prob{S_{n,j} - s_{j} \leq - \frac{s_{j}}{2}} \leq \frac{(C+2)s_{j}^2}{(C+2)s_{j}^2 + s_{j}^2/4} = \frac{4(C+2)}{4(C+2) + 1}. 
\]
Thus, by independence of the $S_{n,j}$ for different $j$, and non-negativity of the summands,
\begin{linenomath}
\begin{align*}
\Prob{\sum_{i=0}^{n-1} U_{n,i} \leq M/2} & = \Prob{\sum_{j=1}^{k} S_{n,j} \leq M/2} \\ & \leq \Prob{\forall j\in\left\{1, \ldots, k\right\} \; S_{n,j} \leq M/2} \leq \left(\frac{4(C+2)}{4(C+2)+1}\right)^{k} \leq \eps. 
\end{align*}
\end{linenomath}
\end{proof}
\begin{proof}[Proof of Claim~\ref{clm:div-sum-ineq}] 
Since the truth of the statement for $M$ large implies the statement for all smaller $M$, without loss of generality suppose $M > 2C$. Then,  given $k \in \mathbb{N}$, we construct the sequence $(n_{i})_{i =0, 1, \ldots k}$ iteratively. First, since $\sum_{\ell =1}^{\infty} u_{\ell} = \infty$, we may choose $n_1$ sufficiently large that 
\begin{equation} \label{eq:initial-sum}
s_{1} = \sum_{\ell=0}^{n_1} u_{\ell} > 4^{k} M. 
\end{equation}
Next, for $j \in \{2, \ldots, k\}$ we set 
\[
n_{j} := \inf_{\ell \in \mathbb{N}}\left\{\ell > n_{j-1} + 1 \colon \sum_{i=n_{j-1} + 1}^{\ell} u_{i} > \frac{s_{j-1}}{2} - 2C\right\}. 
\]
Note that, by construction, since each $u_{\ell} \in (0, C]$, we have 
\begin{equation} \label{eq:bound-on-sj}
s_{j} \in 
\bigg(\frac{s_{j-1}}{2} - 2C, \frac{s_{j-1}}{2} - C\bigg].
\end{equation}
Moreover, if $s_{j-1} > 4^{k-j+1}M > 4M$, since $M > 2C$ we have 
\[
s_{j} > \frac{s_{j-1}}{2} - 2C > \frac{s_{j-1}}{4} > 4^{k-j} M.
\]
Thus we have $s_{j} > M$ for all $j \in \left\{1, \ldots, k\right\}$, implying the first part of~\eqref{eq:means-dom}. Finally, for $j < k$, using the fact that by~\eqref{eq:bound-on-sj} for $\ell \in \left\{j+1, \ldots, k\right\}$
\[
s_{\ell} + C \leq \frac{s_{\ell - 1}}{2} < \frac{s_{\ell - 1} + C}{2} \leq \frac{s_{\ell - 2}}{4} \leq \cdots \leq \frac{s_{j}}{2^{\ell - j}}, 
\]
we have
\[
s_{j} > \left(1 - \frac{1}{2^{k-j}}\right)s_{j} = \sum_{\ell=j+1}^{k}\frac{s_{j}}{2^{\ell - j}} \geq \left(\sum_{\ell = j+1}^{k-1} s_{\ell} \right) + s_{k} + C = C+ \sum_{\ell=j+1}^{k} s_{\ell},
\]
as required for the second part of~\eqref{eq:means-dom}. 
\end{proof}

\section{Proof of Theorem~\ref{thm:rate-of-growth}}
\label{sec:growth-scale}
In the proof of Theorem~\ref{thm:rate-of-growth}, note that
Equation~\eqref{eq:rate-of-growth-mean-comparison} implies that there exists constants $s_1, s_2 > 0$ such that, for all $n \in \mathbb{N}_0$
\begin{equation} \label{eq:rate-of-growth-mean-comparison-2}
s_1 < \frac{C_{n}}{\tilde{M}^{(\gamma)}_{n}} < s_2. 
\end{equation} 
 \subsection{Proofs of sufficiency of conditions of Theorem~\ref{thm:rate-of-growth}}
 We now prove the sufficiency of conditions appearing in Theorem~\ref{thm:rate-of-growth}.

\begin{proof}[Proof of sufficiency of conditions of Theorem~\ref{thm:rate-of-growth}]
First, since~\eqref{eq:rate-of-growth-mean-comparison} is satisfied, there exists $s_1$ such that for all $n\in \mathbb{N}_0$ we have $C_{n}/\tilde{M}^{(\gamma)}_{n} \leq s_2$. It follows that for all, $n \in \mathbb{N}_0$ $\widehat{Y}^{(\gamma)}_{n} \leq s_{2} \gamma \tilde{M}^{(\gamma)}_{n+1}$ almost surely. Thus by Theorem~\ref{thm:unif-integrable} with $K = \gamma s_2$, $\tilde{Z}^{(\gamma)}_{n}/\tilde{M}^{(\gamma)}_{n}$ converges in $L^1$ to a non-degenerate, finite random variable $\tilde{W}$, thus, \[\Prob{0 < \lim_{n \to \infty} \tilde{Z}^{(\gamma)}_{n}/\tilde{M}^{(\gamma)}_{n} < \infty} > 0.\] Again, by~\eqref{eq:rate-of-growth-mean-comparison-2}, this implies that 
\begin{equation} \label{eq:rate-of-growth-for-truncated}
\Prob{0 < \liminf_{n \to \infty} \tilde{Z}^{(\gamma)}_{n}/C_n \leq \limsup_{n \to \infty} \tilde{Z}^{(\gamma)}_{n}/C_n <  \infty} > 0. 
\end{equation}
For the final part of the proof, we exploit the fact that, for a sequence of numbers $(\alpha_{n})_{n \in \mathbb{N}_0} \in (0,1)^{\mathbb{N}_0}$, we have 
\begin{equation} \label{eq:product-identity}
    \prod_{j=0}^{\infty} \alpha_{j} > 0 \quad \text{if and only if} \quad \sum_{j=0}^{\infty} (1 - \alpha_{j}) < \infty. 
\end{equation}
Now, suppose that, for some $a,b > 0$, $(\mathcal{A}_{n}(a,b))_{n \in \mathbb{N}_0}$,  denote the events 
\begin{equation} \label{eq:rate-at-step-n}
\mathcal{A}_{n}(a,b) := \left\{a C_n < \tilde{Z}^{(\gamma)}_{n}  < b C_{n} \right\}. 
\end{equation}
Then, by~\eqref{eq:rate-of-growth-for-truncated}, for some $a, b > 0$ we have $\Prob{\bigcap_{n=0}^{\infty} \mathcal{A}_{n}(a,b)} > 0$. Now, we construct $(Z_{n})_{n \in \mathbb{N}_0}$ by first sampling $(\tilde{Z}^{(\gamma)}_{n})_{n \in \mathbb{N}_0}$ and then, using auxiliary independent random variables, potentially adding extra individuals in every generation to form $(Z_{n})_{n \in \mathbb{N}_0}$. Let 
\begin{equation} \label{eq:rate-at-step-n-1}
\mathcal{B}_{n} := \left\{Z_{n} = \tilde{Z}^{(\gamma)}_{n}\right\}. 
\end{equation}
Then note that disagreement occurs if at least one individual in $\tilde{Z}^{(\gamma)}_{n}$ has an extra child in $Z_{n+1}$, hence 
\begin{linenomath*}
\begin{align*}
& \Prob{\mathcal{B}^{c}_{n+1} \, \bigg | \, (\tilde{Z}^{(\gamma)}_{n})_{n \in \mathbb{N}_0}, \bigcap_{j=0}^{n}\mathcal{B}_j} \mathbf{1}_{\left\{\bigcap_{\ell=0}^{\infty} \mathcal{A}_{\ell}(a,b) \right\}} \\ & \hspace{1cm} \leq \Prob{\exists i \in \left\{1, \ldots, \lfloor b C_n \rfloor\right\}\colon \; X_{n,i}\mathbf{1}_{\left\{X_{n,i} > \gamma C_{n+1}\right\}} > 0} \leq bC_{n} \Prob{X_{n} > \gamma C_{n+1}}, 
\end{align*}
\end{linenomath*}
thus by the chain rule, summability of the right-hand side in $n$, and~\eqref{eq:product-identity}, we have 
\begin{linenomath*}
\begin{align*}
& \Prob{\bigcap_{j=0}^{\infty} \mathcal{B}_{j} \, \bigg | \, (\tilde{Z}^{(\gamma)}_{n})_{n \in \mathbb{N}_0}}\mathbf{1}_{\left\{\bigcap_{\ell=0}^{\infty} \mathcal{A}_{\ell}(a,b) \right\}} = \prod_{n=0}^{\infty} \left(1 - \Prob{\mathcal{B}^{c}_{n} \, \bigg | \, (\tilde{Z}^{(\gamma)}_{n})_{n \in \mathbb{N}_0}, \bigcap_{j=0}^{n-1}\mathcal{B}_j} \right) \mathbf{1}_{\left\{\bigcap_{\ell=0}^{\infty} \mathcal{A}_{\ell}(a,b) \right\}} > 0,  
\end{align*}
\end{linenomath*}
with positive probability. Thus, taking expectations over both sides, we have 
\[
\Prob{\forall n \in \mathbb{N}_{0}\colon Z_{n} = \tilde{Z}^{(\gamma)}_{n} \, \text{ and  } \, a C_n < \tilde{Z}^{(\gamma)}_{n}  < b C_{n}} = \Prob{\bigcap_{n=0}^{\infty} \left(\mathcal{A}_{n}(a,b) \cap \mathcal{B}_{n}\right)} > 0.
\]

Finally, when the limit exists in~\eqref{eq:rate-of-growth-mean-comparison}, the same argument applies to show that $(C_n)_{n \in \mathbb{N}_0}$ is a rate of growth for $(Z_{n})_{n \in \mathbb{N}_0}$.  
\end{proof}
 \subsection{Proofs of necessity of conditions of Theorem~\ref{thm:rate-of-growth}}
To prove necessity of the conditions, we first need the following proposition, whose proof we delay to Section~\ref{sec:annoying-part}. 
\begin{prop} \label{prop:annoying-part}
Let $(Z_n)_{n\geq 0}$ be a BPVE, with a growth scale $(C_{n})_{n \in \mathbb{N}_0}$. Then, for every $\gamma>0$, 
\begin{equation} \label{eq:most-annoying-part}
\liminf_{m \to \infty}
\prod_{n=0}^{m}
\frac{C_{n+1}}
{C_n\,
\mathbb E\left[
X_n\mathbf 1_{\{X_n\leq \gamma C_{n+1}\}}
\right]}
>0.
\end{equation}
\end{prop}

\begin{proof}[Proof of necessity of conditions of Theorem~\ref{thm:rate-of-growth}] If $(C_{n})_{n \in \mathbb{N}}$ were a growth scale of $(Z_{n})_{n \in \mathbb{N}}$ then, for some $a, \gamma > 0$
\[
\Prob{\forall n \in \mathbb{N}_0\colon a C_{n} < Z_{n} < \gamma C_{n}} > 0,
\]
and since the occurrence of this event implies that $Z_n$ agrees with $\tilde{Z}^{(\gamma)}_{n}$, for some $a > 0$ 
\begin{equation} \label{eq:gamma-trunc-for-converse}
\Prob{\forall n \in \mathbb{N}_0 \colon Z_{n} = \tilde{Z}^{(\gamma)}_{n}, a C_{n} < Z^{(\gamma)}_{n} < \gamma C_{n}} > 0.
\end{equation}
Suppose that Equation~\eqref{eq:truncation-determines-original} is not satisfied for this choice of $\gamma$. Using the notation from~\eqref{eq:rate-at-step-n} and~\eqref{eq:rate-at-step-n-1}, we have 
\[
\Prob{\mathcal{A}_{n+1}(a,\gamma) \cap  \mathcal{B}_{n+1} \, \bigg | \, \bigcap_{j=0}^{n} (\mathcal{A}_{j}(a,\gamma) \cap \mathcal{B}_{j})} \leq (1-\Prob{X_{n} > \gamma C_{n+1}})^{a C_n},
\]
and thus, by the by the chain rule, and the bound $1-x \leq e^{-x}$ we have
\begin{linenomath}
    \begin{align*}
        \Prob{\bigcap_{n=0}^{\infty} \left(A_{n}(a,b) \cap \mathcal{B}_{n} \right)} & = \prod_{n=0}^{\infty} \Prob{\mathcal{A}_{n+1}(a,b) \cap \mathcal{B}_{n+1} \, \bigg | \, \bigcap_{j=0}^{n} \left( \mathcal{A}_{j}(a,b) \cap \mathcal{B}_{j} \right)} \\ & \leq e^{-a \sum_{n=0}^{\infty} C_{n} \Prob{X_{n} > \gamma C_{n+1}}} = 0,
    \end{align*}
\end{linenomath}
contradicting~\eqref{eq:gamma-trunc-for-converse}. This implies~\eqref{eq:truncation-determines-original}. 
Next, Equation~\eqref{eq:gamma-trunc-for-converse} implies that $(C_{n})_{n \in \mathbb{N}_0}$ is a growth scale of $(\tilde{Z}^{(\gamma)}_{n})_{n \in \mathbb{N}_0}$. Thus, the fact that the left-side of~\eqref{eq:rate-of-growth-mean-comparison} is positive follows from Proposition~\ref{prop:annoying-part}. Moreover, since the non-negative martingale $(\tilde{Z}^{(\gamma)}_{n}/\tilde{M}^{(\gamma)}_{n})_{n \in \mathbb{N}_{0}}$ has an almost surely finite limit $\tilde{W}^{(\gamma)}$, we have, 
\begin{equation} \label{eq:liminf-blah}
\liminf_{n \to \infty} \frac{\tilde{Z}^{(\gamma)}_{n}}{C_{n}} = \frac{\tilde{W}^{(\gamma)}}{\limsup_{n \to \infty} \frac{C_{n}}{\tilde{M}^{(\gamma)}_{n}}},
\end{equation}
from which, in order for $(C_{n})_{n \in \mathbb{N}_0}$ to be a growth-scale, we require the right-side of~\eqref{eq:rate-of-growth-mean-comparison} to be finite. 
Finally, since~\eqref{eq:rate-of-growth-mean-comparison} is satisfied, $(C_{n})_{n \in \mathbb{N}_0}$ is a growth-scale of $(\tilde{Z}^{(\gamma)}_{n})_{n \in \mathbb{N}_0}$ if and only if $(\tilde{M}^{(\gamma)}_{n})_{n \in \mathbb{N}_0}$ is a rate of growth of $(\tilde{Z}^{(\gamma)}_{n})_{n \in \mathbb{N}_0}$. But now, again, if $s_2 := \sup_{n \in \mathbb{N}_0} C_{n}/\tilde{M}^{(\gamma)}_{n}$, by applying Theorem~\ref{thm:unif-integrable} with $K = \gamma s_2$, this is true if and only if~\eqref{eq:truncation-converges} is satisfied. 

Finally the same logic carries through to show that if $(C_{n})_{n \in \mathbb{N}_0}$ is a rate of growth of $(Z_{n})_{n \in \mathbb{N}_0}$, it is also a rate of growth of $(Z^{(\gamma)}_{n})_{n \in \mathbb{N}_0}$. But then, by~\eqref{eq:rate-of-growth-mean-comparison} it must be the case that 
\[
\left\{0 < \lim_{n \to \infty} \frac{Z^{(\gamma)}_{n}}{C_{n}} < \infty\right\} \subseteq \left\{0 < \lim_{n \to \infty} \frac{Z^{(\gamma)}_{n}}{\tilde{M}^{(\gamma)}_{n}} < \infty\right\}, 
\]
hence, the limit in~\eqref{eq:rate-of-growth-mean-comparison} must exist. 
\end{proof}

\subsection{Proof of Proposition~\ref{prop:annoying-part}} \label{sec:annoying-part}
\begin{proof}[Proof of Proposition~\ref{prop:annoying-part}]
Let $\gamma > 0$ be given. Note that, since $(C_{n})_{n \in \mathbb{N}_0}$ is a growth scale, there exists $a, b \in (0,\infty)$ such that the event
\begin{equation} \label{eq:def-A}
A:=
\left\{
aC_n\leq Z_n\leq bC_n \: \forall n \in \mathbb{N}_0 \right\}
\end{equation}
has positive probability. Now, we set 
\[
f_n(s):=\mathbb E[s^{X_n}], \quad \text{and, for $m > n$, write} \quad F_{n,m}:=f_n\circ f_{n+1}\circ\cdots\circ f_{m-1}.
\]
(We follow the convention that $f_{n}(1) := \lim_{s \uparrow 1} f_{n}(s)$.)
For a given $t > 0$, for each $n$, the sequences $(F_{n,m}(e^{-t/C_{m}})$ and $(F_{n,m}(e^{-2t/C_{m}})$ take values in $[0,1]$. Thus, by a diagonalisation argument, there exists a subsequence $(m_{j})_{j \in \mathbb{N}}$ such that, for every $n \in \mathbb{N}_0$, the limits
\[
x_{n} := \lim_{j \to \infty} F_{n,m_j}\left(e^{-t/C_{m_j}}\right) \quad \text{and} \quad y_n := \lim_{j \to \infty} F_{n,m_j}\left(e^{-2t/C_{m_j}}\right)
\]
exist in $[0,1]$. Moreover, by continuity of $f_n$, $x_n=f_n(x_{n+1})$ and $y_n=f_n(y_{n+1})$. Thus, if $(\mathcal{F}_{n})_{n \in \mathbb{N}_0}$ denotes the natural filtration associated with $(Z_{n})_{n\in \mathbb{N}_0}$ 
\[
\E{x_{n+1}^{Z_{n+1}}\mid\mathcal F_n}= \E{x_{n+1}^{\sum_{j=1}^{Z_{n}} X_{n,j}}\mid Z_n} = 
\prod_{i=1}^{Z_n}\mathbb E[x_{n+1}^{X_{n}}] =
f_n(x_{n+1})^{Z_n} =
x_n^{Z_n},
\]
and similarly for $(y_n^{Z_{n}})_{n \in \mathbb{N}_0}$. Therefore, $(x_n^{Z_{n}})_{n \in \mathbb{N}_0}$ and $(y_n^{Z_{n}})_{n \in \mathbb{N}_0}$ are bounded martingales.

We now have the following two claims, whose proof we defer to the end of the section:
\begin{clm} \label{clm:x-n-y-n}
    There exists $d_1, d_2, d_3, t > 0$, with $d_1 < \gamma^{-1}$ such that, with $(x_{n})_{n \in \mathbb{N}_0}$ and $(y_{n})_{n \in \mathbb{N}_0}$ as defined above, there exists $N_1 \in \mathbb{N}$ such that for all $n \geq N_1$
    \begin{equation} \label{eq:seneta-rate}
        \frac{d_1}{C_{n}} \leq 1 - x_n \leq \frac{\gamma^{-1}}{C_n} \quad \text{and} \quad \frac{d_2}{C_{n}} \leq x_n - y_n \leq \frac{d_3}{C_n}. 
    \end{equation}
    Moreover, with 
    \begin{equation} \label{eq:u-n-def}
    U_n := \frac{1-x_n}{1-x_{n+1}} = \frac{1 - f_{n}(x_{n+1})}{1 - x_{n+1}}\quad \text{and} \quad V_{n} := \frac{x_{n} - y_{n}}{x_{n+1} - y_{n+1}} = \frac{f_{n}(x_{n+1}) - f_{n}(y_{n+1})}{x_{n+1} - y_{n+1}}
    \end{equation} we have 
    \begin{equation} \label{eq:v-over-u-product-conv}
        \sum_{n=N_{1}}^{\infty} \left(1 - \frac{V_{n}}{U_{n}} \right) < \infty.  
    \end{equation}
\end{clm}
\begin{clm} \label{clm:a-k-b-k}
    For  $0 < y < x < 1$, $k \in \mathbb{N}_0$ set \[A_{k} := \frac{1-x^{k}}{1-x} \quad \text{and} \quad B_{k} := \frac{x^{k} - y^{k}}{x-y}.\] Then, for all $k \in \mathbb{N}_0$ 
    \begin{equation} \label{eq:trunc-comparison-bound}
        \left|A_{k} - k \mathbf{1}_{k \leq (1-x)^{-1}} \right| \leq 3 (A_{k} - B_{k}). 
    \end{equation}
\end{clm}
Given the two claims above, by~\eqref{eq:seneta-rate}, Claim~\ref{clm:a-k-b-k} is applicable for all $n \geq N_1$. By applying Claim~\ref{clm:a-k-b-k} with $y = y_{n+1}$ and $x = x_{n+1}$, and the definitions of $A_{k}$ and $B_{k}$ adapted accordingly, note that
\begin{equation} \label{eq:adapting-a-k-b-k-to-exp}
\E{A_{X_{n}}} = \frac{1 - f_{n}(x_{n+1}) }{1 - x_{n+1}} \stackrel{\eqref{eq:u-n-def}}{=} U_{n} \quad \text{ and } \quad \E{B_{X_{n}}} = \frac{f_{n}(x_{n+1}) - f_{n}(y_{n+1}) }{x_{n+1} - y_{n+1}} \stackrel{\eqref{eq:u-n-def}}{=} V_{n}. 
\end{equation}
Thus, by Jensen's inequality and~\eqref{eq:trunc-comparison-bound}, if we set $R_{n} := (1-x_{n+1})^{-1}$ 
\[
\left|U_{n} - \E{X_{n} \mathbf{1}_{X_{n} \leq R_{n}}} \right| \leq \E{\left|A_{X_n} - X_{n} \mathbf{1}_{X_{n} \leq R_{n}}\right|} \stackrel{\eqref{eq:trunc-comparison-bound}}{\leq} 3 \E{A_{X_{n}} - B_{X_{n}}} \stackrel{\eqref{eq:adapting-a-k-b-k-to-exp}}{=} 3\left(U_{n} - V_{n}\right),
\]
thus, dividing both sides by $U_n$ and summing over $n$, 
\begin{equation} \label{eq:abs-summ}
 \sum_{n=N_1}^{\infty} \left|1 - \frac{\E{X_{n} \mathbf{1}_{X_{n} \leq R_{n}}}}{U_n} \right| \leq 3 \sum_{n = N_1}^{\infty} \left(1 - \frac{V_n}{U_n}\right) \stackrel{\eqref{eq:v-over-u-product-conv}} < \infty.
\end{equation}
Since absolute summability implies summability, we deduce that 
\begin{linenomath}
\begin{align} \label{eq:conv-product-one}
\nonumber\limsup_{m \to \infty}\prod_{n=N_1}^{m} \frac{\E{X_{n} \mathbf{1}_{X_{n} \leq R_n}}}{U_n} & = \limsup_{m \to \infty}\prod_{n=N_1}^{m}\left(1 -   \left(1 - \frac{\E{X_{n} \mathbf{1}_{X_{n} \leq R_n}}}{U_n}\right) \right) \\ &\leq \exp{\left(\sum_{n=N_1}^{\infty} \left(1 - \frac{\E{X_{n} \mathbf{1}_{X_{n} \leq R_{n}}}}{U_n} \right) \right)} < \infty. 
\end{align}
\end{linenomath}
Thus, since the $U_{n}$ product telescopes, 
\begin{linenomath}
    \begin{align} \label{eq:second-last}
\nonumber        \liminf_{m \to \infty} \prod_{n=N_1}^{m}
\frac{C_{n+1}}{C_n\E{X_{n} \mathbf{1}_{X_{n} \leq R_n}}} & = \liminf_{m \to \infty} \prod_{n=N_1}^{m}\frac{C_{n+1}}{C_n U_n} \prod_{n=N_1}^{m}
\frac{U_n}{\E{X_{n} \mathbf{1}_{X_{n} \leq R_{n}}}}\\
& \stackrel{\eqref{eq:u-n-def}}{=} \liminf_{m \to \infty}
\frac{C_m(1-x_m)}
{C_N(1-x_N)}
\prod_{n=N}^{m-1}
\frac{U_n}{\E{X_{n} \mathbf{1}_{X_{n} \leq R_{n}}}} \stackrel{\eqref{eq:seneta-rate}, \eqref{eq:conv-product-one}}{>} 0.
    \end{align}
\end{linenomath}
Finally, since again by~\eqref{eq:seneta-rate} we have $R_{n} = (1-x_{n+1})^{-1} > \gamma C_{n+1}$, we deduce~\eqref{eq:most-annoying-part} from~\eqref{eq:second-last}. 
\end{proof}
It remains to complete the proofs of the claims. 
\begin{proof}[Proof of Claim~\ref{clm:x-n-y-n}]
Given $\gamma > 0$, with $a, b$ as defined in~\eqref{eq:def-A}, choose $t \in (0, a \gamma^{-1}/(2b))$ and let $(x_{n})_{n \in \mathbb{N}_0}$ and $(y_{n})_{n \in \mathbb{N}_0}$ be defined with this $t$.   Given the sequence $(m_{j})_{j \in \mathbb{N}}$ defining $(x_{n})_{n \in \mathbb{N}_0}$ and $(y_{n})_{n \in \mathbb{N}_0}$, note that on the event $A$,
\[
e^{-tb}
\leq
e^{-tZ_{m_j}/C_{m_j}}
\leq
e^{-ta},
\]
so that, taking conditional expectations of both sides 
\[
e^{-tb} \Prob{A \mid \mathcal{F}_n} \leq \underbrace{\E{e^{-tZ_{m_j}/C_{m_j}} \, \bigg | \, \mathcal{F}_n}}_{ = \, \left(F_{n, m_{j}}\left(e^{-t/C_{m_j}}\right)\right)^{Z_{n}}} \leq e^{-ta} \Prob{A \mid \mathcal{F}_n} + (1 - \Prob{A \mid \mathcal{F}_n})
\]
Hence, by taking limits in $j$, 
\[
e^{-tb}\mathbb P(A\mid\mathcal F_n)
\leq
x_{n}^{Z_n}
\leq
e^{-ta}\mathbb P(A\mid\mathcal F_n)
+
1-\mathbb P(A\mid\mathcal F_n).
\]
Letting $n\to\infty$, and applying similar reasoning to $(y_n)_{n \in \mathbb{N}_0}$, by the L\'evy zero-one law, we obtain
\[
e^{-tb}\leq x_\infty\leq e^{-ta} \quad \text{ and } \quad 
e^{-2tb}\leq y_\infty\leq e^{-2ta}
\quad \text{ on } A,
\]
where $x_{\infty}, y_{\infty}$ denote the almost sure limits of the bounded martingales $(x_n^{Z_{n}})_{n \in \mathbb{N}_0}$ and $(y_n^{Z_{n}})_{n \in \mathbb{N}_0}$.
Moreover, setting $c_0 := \min_{u\in[a,b]}
\left(e^{-tu}-e^{-2tu}\right)>0$,
we have
\[
e^{-tZ_{m_j}/C_{m_j}}
-
e^{-2tZ_{m_j}/C_{m_j}}
\geq c_0\mathbf 1_A, \quad \text{and thus} \quad 
x_\infty-y_\infty \geq c_0
\quad \text{ on } A.  
\]
It follows that there exists a positive probability event $B \subseteq A$ and $N_1 \in \mathbb{N}$ such that, for all $n \geq N_1$
\begin{equation} \label{eq:bounds-on-b}
e^{-2tb} \leq x_{n}^{Z_{n}} \leq e^{-ta/2}, \quad y_{n}^{Z_n} \geq e^{-3tb} \quad \text{ and } \quad x_{n}^{Z_n} - y_{n}^{Z_n} \geq \frac{c_0}{2} \quad \text{ on } B. 
\end{equation}
Now, by noting that on $B$ we have $aC_n \leq Z_{n} \leq bC_n$, Equation~\eqref{eq:seneta-rate} now follows from~\eqref{eq:bounds-on-b} by the choice of $t$ and exploiting the elementary inequalities for $x > y > 0$ that $n y^{n-1} (x-y) \leq x^{n} - y^{n} \leq n x^{n-1} (x-y)$, and $x \geq 1 -e^{-x} \geq \frac{x}{1+x}$. (Note that these are deterministic bounds since $x_n$ and $y_n$ are deterministic).

Next, note that by convexity of the function $f_{n}$, the mean value theorem, and the fact $y_n < x_n < 1$, we have $V_{n} \leq U_{n}$, and since the product telescopes
\[
\prod_{n=N_1}^{\infty} \frac{V_{n}}{U_{n}} = \lim_{m \to \infty} \prod_{n=N_1}^{\infty} \frac{V_{n}}{U_{n}} = \lim_{m \to \infty} \frac{x_{N_{1}} - y_{N_{1}}}{1-x_{N_{1}}} \times \frac{1 - x_{m+1}}{x_{m+1} - y_{m+1}} \stackrel{\eqref{eq:seneta-rate}}{\geq} \frac{d_1 d_2 \gamma}{d_3}. 
\]
Equation~\eqref{eq:v-over-u-product-conv} now follows from properties of infinite products (Equation~\eqref{eq:product-identity}). 
\end{proof}

Finally, we prove Claim~\ref{clm:a-k-b-k}. 
\begin{proof}[Proof of Claim~\ref{clm:a-k-b-k}]
For $k=0$ note that $A_0 = B_0 = 0$, whilst for $k=1$, since we always have $(1-x)^{-1} > 1$, again both sides of~\eqref{eq:trunc-comparison-bound} are $0$. It suffices to consider $k \geq 2$. 
Since the function $u\mapsto u^k$ is convex and $y<x$, we have $A_{k} \geq B_{k}$, 
\begin{equation} \label{eq:b-k-bound}
B_k=\frac{x^k-y^k}{x-y}\leq kx^{k-1},
\qquad\text{hence}\qquad
A_k-B_k\geq A_k-kx^{k-1}.
\end{equation}
If $k(1-x) \leq 1$, then 
\begin{equation} \label{eq:k-a-k}
k - A_{k} = \sum_{j=0}^{k-1} (1- x^{j}) \leq (1-x) \sum_{j=0}^{k-1} j = \frac{(1-x)k(k-1)}{2}, 
\end{equation}
where we use $1 - x^{j} \leq j(1-x)$. 
On the other hand, applying the mean-value theorem
\begin{linenomath}
\begin{align} \label{eq:ingredient-1}
\nonumber  A_{k} - kx^{k-1} & = \frac{k}{1-x}\int_{x}^{1} u^{k-1} - x^{k-1} \dd u \geq \frac{k(k-1)x^{k-2}}{1-x} \int_{x}^{1} (u-x) \dd u \\ & = \frac{k(k-1)(1-x)x^{k-2}}{2}   \geq \frac{k(k-1)(1-x)\left(1 - 1/k\right)^{k-2}}{2}.  
\end{align}
\end{linenomath}
Note that for all $k \geq 2$ we have $(1-1/k)^{k-2} > e^{-1}$, hence by~\eqref{eq:b-k-bound},~\eqref{eq:k-a-k} and~\eqref{eq:ingredient-1}, for $k(1-x) \leq 1$ we have   
\begin{equation} \label{eq:case-1-annoying}
k-A_{k} \leq e(A_{k} - B_{k}).
\end{equation}
On the other hand, suppose that $k(1-x)>1$. Then $x<1-1/k$. We then have
\[
\frac{kx^{k-1}}{A_k}
=
\frac{kx^{k-1}}{\sum_{j=0}^{k-1}x^j}
=
\frac{k}{\sum_{j=0}^{k-1}x^{-j}} = \frac{kx^{k-1}(1-x)}{1-x^{k}} < \frac{k (1-1/k)^{k-1}1/(k-1)}{1 - (1-1/k)^{k}} \leq \frac{2}{3}, 
\]
where in the inequality we use the fact (from the second equality) that the term on the left is increasing on $x \in (0,1)$. 
Therefore, by~\eqref{eq:b-k-bound}, for $k(1-x) > 1$
\[
3(A_{k} - B_{k}) \geq A_{k},
\]
hence, in all cases
$\left|A_k-k\mathbf 1_{{k\leq(1-x)^{-1}}}\right|
\leq
3(A_k-B_k)$, as required in~\eqref{eq:trunc-comparison-bound}.

\end{proof}

\subsection{Proof of Proposition~\ref{prop:gamma-arbitrary}}
In this section we prove Proposition~\ref{prop:gamma-arbitrary}, which implies the truth of Remark~\ref{rmq:gamma-arbitrary}. 
\begin{prop} \label{prop:gamma-arbitrary}
The following are equivalent:
\begin{enumerate}
    \item \label{item:gamma-arbitrary-1} Equations~\eqref{eq:truncation-determines-original},~\eqref{eq:rate-of-growth-mean-comparison} and~\eqref{eq:truncation-converges}, are satisfied for some $\gamma > 0$.
    \item \label{item:gamma-arbitrary-2} Equations~\eqref{eq:truncation-determines-original},~\eqref{eq:rate-of-growth-mean-comparison} and~\eqref{eq:truncation-converges} are satisfied for all $\gamma' > 0$ satisfying $\E{Y^{(\gamma')}_{n}} > 0$ for all $n \in \mathbb{N}_0$.
\end{enumerate}
\end{prop}

\begin{proof}
    We need only prove that Item~\ref{item:gamma-arbitrary-1} implies Item~\ref{item:gamma-arbitrary-2} since the other direction is immediate. First we prove Item~\ref{item:gamma-arbitrary-2} when $\gamma' < \gamma$. In this case, since $\E{Y^{(\gamma')}_{n}} = \E{X_{n}\mathbf{1}_{X_{n} \leq \gamma'C_{n+1}}} > 0$, it must be the case that $\gamma'C_{n+1} \geq 1$. Then, for~\eqref{eq:truncation-determines-original}, note that 
    \begin{linenomath}
    \begin{align*}
    \sum_{n=0}^{\infty} C_{n} \Prob{X_{n} > \gamma' C_{n+1}} & \leq \sum_{n=0}^{\infty} C_{n} \Prob{X_{n} > \gamma C_{n+1}} + \sum_{n=0}^{\infty} C_{n} \Prob{Y^{(\gamma)}_{n} \geq \gamma' C_{n+1} + 1}
    \\ & \leq \sum_{n=0}^{\infty} C_{n} \Prob{X_{n} > \gamma C_{n+1}} + \sum_{n=0}^{\infty} \frac{C_{n}\E{\left(\widehat{Y}^{(\gamma)}_{n} - 1\right)\mathbf{1}_{\widehat{Y}^{(\gamma)}_{n} \geq \gamma'C_{n+1} + 1}}}{(\gamma'C_{n+1})^2} < \infty, 
    \end{align*}
\end{linenomath}
where the last line follows from the fact that~\eqref{eq:truncation-converges} is satisfied for $\gamma$. Next, writing 
\begin{equation} \label{eq:means-inf-product}
\frac{\E{Y^{(\gamma')}_{n}}}{\E{Y^{(\gamma
)}_{n}}} \geq 1 - \frac{\gamma C_{n+1} \Prob{Y^{(\gamma)}_{n} > \gamma' C_{n+1}}}{\E{Y^{(\gamma)}_{n}}}, \quad \text{we have} \quad 1 \geq \prod_{n=0}^{\infty} \frac{\E{Y^{(\gamma')}_{n}}}{\E{Y^{(\gamma
)}_{n}}} > 0 
\end{equation}
by properties of infinite products (\eqref{eq:product-identity}), Equation~\eqref{eq:rate-of-growth-mean-comparison} for $\gamma$ and~\eqref{eq:truncation-determines-original} for $\gamma'$. This implies that $\gamma'$ satisfies~\eqref{eq:truncation-determines-original}. Finally,~\eqref{eq:truncation-converges} follows from monotonicity and the fact $\gamma'$ satisfies~\eqref{eq:truncation-determines-original}.

For $\gamma' > \gamma$,~\eqref{eq:truncation-determines-original} is immediate from monotonicity, whilst a similar argument to~\eqref{eq:means-inf-product}, but instead with an upper bound, shows that $\gamma'$ satisfies~\eqref{eq:rate-of-growth-mean-comparison}. Since~\eqref{eq:rate-of-growth-mean-comparison} is satisfied, for~\eqref{eq:truncation-converges}, by Lemma~\ref{lem:three-series-positive}, it suffices to show
\[
\sum_{n=0}^{\infty} \frac{\E{\widehat{Y}^{(\gamma')}_{n} - 1}}{C_{n+1}} < \infty. 
\]
But then, again using Lemma~\ref{lem:three-series-positive}, this follows from the fact that 
\begin{linenomath}
\begin{align*}
\sum_{n=0}^{\infty} \frac{\E{\widehat{Y}^{(\gamma')}_{n} - 1}}{C_{n+1}} & \leq \sum_{n=0}^{\infty} \frac{\E{\left(\widehat{Y}^{(\gamma')}_{n} - 1\right)\mathbf{1}_{\widehat{Y}^{(\gamma')}_{n} \leq \gamma C_{n+1}}}}{C_{n+1}} + \sum_{n=0}^{\infty} \frac{(\gamma' C_{n+1})^2 \Prob{\widehat{Y}^{(\gamma)}_{n} > \gamma C_{n+1}}}{C_{n+1}}
\\ & \leq \sum_{n=0}^{\infty} \frac{\E{\left(\widehat{Y}^{(\gamma)}_{n} - 1\right)}}{C_{n+1}} + \sum_{n=0}^{\infty} \frac{(\gamma')^2 C_{n+1} \Prob{X_{n} > \gamma C_{n+1}}}{\E{Y^{(\gamma)}_{n}}} < \infty,
\end{align*}
\end{linenomath}
where the last term is finite since $C_{n+1}/\E{Y^{(\gamma)}_{n}} \leq \kappa C_{n}$ for some $\kappa > 0$ since $\gamma$ satisfies~\eqref{eq:rate-of-growth-mean-comparison}, and moreover, since $\gamma$ satisfies~\eqref{eq:truncation-determines-original}. 
\end{proof}
\section{Proof of Corollary~\ref{cor:lindval}} \label{sec:cor}
\begin{proof}[Proof of Corollary~\ref{cor:lindval}]
In order for $Z_{n}$ to be bounded above and below, the constant sequence $(1)_{n \in \mathbb{N}_0}$ must be a growth scale. We first show that the conditions in Theorem~\ref{thm:rate-of-growth} are equivalent to $\sum_{n=0}^{\infty} \Prob{X_{n} \neq 1} < \infty$. 

Indeed, if by Equation~\eqref{eq:rate-of-growth-mean-comparison} is satisfied, for some $\gamma > 0$, we must have $\tilde{M}^{(\gamma)}_{n}$ bounded above and below. Note that this also implies that $\E{Y^{(\gamma)}}_{n}$ is bounded above and below, hence, 
\[
 \sum_{n=0}^{\infty} \E{\widehat{Y}^{(\gamma)}_{n} - 1} < \infty, \quad \text{which, alongside~\eqref{eq:truncation-determines-original}, holds iff} \quad \sum_{n=0}^{\infty} \Prob{X_{n} \geq 2} < \infty. 
\]
Then, since $\E{Y^{(\gamma)}_{n}} - 1 \leq \gamma \Prob{X_{n} \geq 2} - \Prob{X_{n}=0}$, the inequality $x \leq e^{x-1}$ implies that
\begin{equation} \label{eq:one-side-tail}
0 < \sup_{n \in \mathbb{N}_0}\tilde{M}^{(\gamma)}_{n} \leq \exp\left(\gamma \sum_{n=0}^{\infty}\Prob{X_{n} \geq 2} - \sum_{n=0}^{\infty} \Prob{X_{n} = 0}\right),
\end{equation}
hence we must also have $\sum_{n=0}^{\infty} \Prob{X_{n} = 0} < \infty$. It suffices to show that $\sum_{n=0}^{\infty} \Prob{X_{n} \neq 1} < \infty$ implies that for some $\gamma > 0$ $\lim_{n \to \infty} \tilde{M}^{(\gamma)}_{n}$ exists in $(0, \infty)$. This follows from the infinite product criterion
\[
\prod_{n=0}^{\infty} \alpha_{n} \quad \text{converges iff} \quad \sum_{n=0}^{\infty} \left|1 - \alpha_n\right| < \infty, 
\]
from which the required statement follows from the fact that 
\[
\left|\E{Y^{(\gamma}_{n}} - 1 \right| \leq \gamma \Prob{X_{n} \neq 1}. 
\]
Finally, since $0$ is an absorbing state, the corollary follows from the following claim:
\begin{clm} \label{clm:last}
Let $(Z_n)_{n\geq 0}$ be a BPVE with offspring law $\boldsymbol{\mu}$. Then
\[
    \mathbb P\left(
        \limsup_{n\to\infty} Z_n = \infty
        \text{ and }
        \liminf_{n\to\infty} Z_n < \infty
    \right)=0.
\]
\end{clm}
\end{proof}
\begin{proof}[Proof of Claim~\ref{clm:last}]
For $n\geq 0$, let $q_n$ denote the extinction probability for the process
started from a single particle at generation $n$. Let $\mathcal{S}$ denote the event of survival, that is $\mathcal{S} := \left\{\liminf_{n \to \infty} Z_{n} \geq 1\right\}$ Thus, conditionally on
$Z_n$,
\[
    \Prob{\mathcal{S}^{c} \mid Z_n}
    =
    q_n^{Z_n} \quad \text{hence by Levy's zero-one law} \quad \mathbf{1}_{\mathcal{S}^{c}} = \lim_{n \to \infty} q_{n}^{Z_{n}}.
\]
Hence, 
\begin{equation} \label{eq:last}
    \lim_{n \to \infty} q_{n}^{Z_n} \quad \text{almost surely, on survival}.
\end{equation} Now let $I_n$ be the number of particles alive at generation $n$ whose
descendants survive forever. Conditionally on $Z_n$,
\[
    I_n \sim \operatorname{Bin}(Z_n,1-q_n).
\]
Moreover, $(I_n)_{n\geq 0}$ is nondecreasing, since every particle with an
infinite line of descent must have at least one child with an infinite line of
descent.

We show that, on the event $\{\limsup_{n\to \infty} Z_n=\infty\}$, one has
$\lim_{n \to \infty} I_n = \infty$. Set $\lambda_{n} := Z_{n}(1-q_n)$. First, note that, on this event, along any subsequence $(n_{j})$ for which $\lim_{j \to \infty} Z_{n_{j}} = \infty$, we have \[\liminf_{j \to \infty} \lambda_{n_{j}} = \liminf_{j \to \infty} Z_{n_{j}} (1- q_{n_{j}}) = \infty. \] 
Indeed, if this were not the case, if $C$ denotes an eventual upper bound on a subsequence, we would have 
\begin{equation} \label{eq:upper-bound-rate}
\limsup_{n \to \infty} q_{n}^{Z_{n}} = \limsup_{n \to \infty} (1- (1-q_{n}))^{Z_{n}} \geq \limsup_{n \to \infty} \left( 1 - \frac{C}{Z_{n}}\right)^{Z_{n}} = e^{-C},  
\end{equation}
contradicting~\eqref{eq:last}. Note also that, with $(\mathcal{F}_{n})_{n \in \mathbb{N}_0}$ denoting the filtration generated by the process $\E{I_{n}\mid \mathcal{F}_{n}} = \lambda_{n}$ and $\Var{I_{n}\mid \mathcal{F}_{n}} \leq \lambda_{n}$. By Cantelli's inequality, for any $K \in \mathbb{N}$
\[
\Prob{I_{n} \leq K \mid \mathcal{F}_{n}} \leq \frac{\lambda_{n}}{\lambda_{n} + (\lambda_{n} - K)^2}.
\]
Since the right-side tends to $0$ along a subsequence where $\lambda_{n} \rightarrow \infty$, and $I_{n}$ is non-decreasing, by Levy's zero-one law we must have $\lim_{n\to\infty} I_{n} = \infty$. Since $Z_{n} \geq I_{n}$, this implies $\lim_{n \to \infty} Z_{n} = \infty$. 
\end{proof}

\section*{Acknowledgements}
TI is funded by Deutsche Forschungsgemeinschaft (DFG) through DFG Project no. $443759178$.

\bibliographystyle{abbrv}
\bibliography{refs}

\end{document}